\documentclass{article}
\usepackage{amssymb}
\usepackage{mathrsfs}
\usepackage{amsmath}
\usepackage{amsthm}

\usepackage{tikz}

\newcommand{\semicircle}[3]{\draw[thick] (#1+#3,#2) arc [start angle=0, end angle=180, radius=#3];}

\usepackage{todonotes}

\newtheorem{thm}{Theorem}

\theoremstyle{definition}

\newtheorem{lem}[thm]{Lemma}
\newtheorem{prop}[thm]{Proposition}
\newtheorem{cor}[thm]{Corollary}

\newtheorem{eg}[thm]{Example}

\theoremstyle{remark}
\newtheorem*{claim}{Claim}

\newcommand{\intersect}{\cap}

\newcommand{\N}{\mathbb{N}}

\newcommand{\R}{\mathbb{R}}

\newcommand{\move}{\text{Move}}

\newcommand{\cl}[1]{\overline{#1}}

\begin{document}
\title{Minimum Topological Group Topologies}

\author{Xiao Chang and Paul Gartside}
\date{October 2015}

\maketitle

\begin{abstract}
A Hausdorff topological group topology on a group $G$ is the minimum (Hausdorff) group topology if it is contained in every Hausdorff group topology on $G$. For every compact metrizable space $X$ containing an open $n$-cell, $n\ge2$, the homeomorphism group $H(X)$ has no minimum  Hausdorff group topology. The homeomorphism groups of the Cantor set and the Hilbert cube have no minimum group topology. For every compact metrizable space $X$ containing a dense open one-manifold, $H(X)$ has the minimum group topology. Some, but not all, oligomorphic groups have the minimum group topology. 
\end{abstract}

\section{Introduction}
A topology $\tau$ on a group $G$ is a \emph{topological group topology} if the group operations on $G$ are $\tau$-continuous. The collection of all topological group topologies on a group $G$ is partially ordered by set-inclusion, with the discrete topology as the maximum element, and the indiscrete topology the minimum. The collection of all \emph{Hausdorff} topological group topologies on $G$ is then a sub partial order of all topological group topologies. This sub order always has a maximum element (the discrete topology) but may or may not have a minimum element (a Hausdorff group topology on $G$ contained in all other Hausdorff group topologies), and also may or may not have minimal elements (a Hausdorff group topology such that no strictly coarser group topology is Hausdorff). Note that the minimum Hausdorff group topology (if it exists) is certainly minimal, but the converse is false in general. 

Minimal (Hausdorff) group topologies have been extensively studied, see the survey \cite{DM} for example. Minimum (Hausdorff) group topologies seem less well understood. Gaughan \cite{Gau} showed that the group $S(X)$ of all permutations of a set $X$ has the minimum group topology,  namely the topology of pointwise convergence. Later the second author and Glyn \cite{GaGl} showed that the group $H(X)$ of all (auto)homeomorphisms of a space $X$, where $X$ is any metrizable one-manifold (with or without boundary), has the minimum group topology. In this case the minimum group topology is the usual compact-open topology on $H(X)$. Recently Megrelishvili and Polev \cite{MP} have extended this last theorem to $H(X)$ for many compact linearly ordered spaces  $X$. 

We continue this line of investigation into the existence, or otherwise, of minimum Hausdorff group topologies. A number of questions raised in \cite{DM} are answered (notably, Questions~2.3 and~4.28), as are Questions~5.1, 5.2(1) and~5.3(1) from \cite{MP}.

First it is established that many homeomorphism groups do not have a minimum Hausdorff group topology (Theorem~\ref{nomin}). In particular for every compact metrizable space $X$ containing an open $n$-cell, $n\ge2$, the homeomorphism group $H(X)$ has no minimum  Hausdorff group topology. Nor does the homeomorphism group of the Cantor set or the Hilbert cube. 

We then show how, in certain circumstances, it is possible to `shrink' the compact-open topology $\tau_k$ on a homeomorphism group $H(X)$ around a closed subset $C$ of $X$ to obtain a new (Hausdorff) group topology $\tau_{k|C}$ on $H(X)$ (Theorem~\ref{thm:taukC}). This shrinking process is derived from an idea of Gamarnik \cite{Gam}. It follows, for example, that the compact-open topology on the homeomorphism group of the Mobius band is not minimal. 

For the purposes of this paper, however,  the real merit of the $\tau_{k|C}$ topology is that when a compact metrizable space $X$ contains a dense open one-manifold then $H(X)$ has the minimum group topology, and that topology is $\tau_{k|C}$, where $C$ is the complement of the one-manifold part of $X$ (Theorem~\ref{thm:taukc_min}). A wide range of spaces satisfy the hypothesis of this theorem. We investigate for these spaces when $\tau_{k|C}$ is equal to $\tau_k$ and when it is strictly smaller. In both cases we give sufficient conditions (see Propositions~\ref{suff_tkcnetk} and~\ref{suff_tkceqtk}, and the accompanying  examples). However a complete classification seems elusive, and we present examples demonstrating the difficulties. Note that when $\tau_{k|C}$ is strictly smaller it follows that the compact-open topology, $\tau_k$ is \emph{not minimal}, but when they are equal we deduce that the compact-open topology is \emph{minimal}. Consequently we derive large classes of compact metric spaces for which we know whether or not the compact-open topology is a minimal  group topology on the homeomorphism group.  

Finally we turn from homeomorphism groups to automorphism groups. For a countable model $M$ of a first order theory, $\mathop{Aut}(M)$ is the group of automorphisms of $M$, and we can topologize it as a subgroup of $S(M)$ with the topology of pointwise convergence. We show that certain oligomorphic automorphism groups, including all strongly homogeneous automorphism groups, have the minimum (Hausdorff) group topology (Theorem~\ref{hom_min}). In most cases this minimum group topology is strictly coarser than the topology of pointwise convergence, and so this latter topology is not minimal. But we also show that the automorphism group of the atomless countable Boolean algebra has no minimum group topology.  

\section{Homeomorphism Groups with no Minimum}

Let $(X,d)$ be a compact metric space. Recall that basic neighborhoods in $H(X)$ of the identity in the pointwise topology have the form 
$B_F^{\tau_p}=\{f \in H(X) : f(x)=x$ for all $x \in F\}$, for some finite subset $F$ of $X$. While basic neighborhoods in $H(X)$ of the identity in the compact-open topology have the form  $B_\epsilon^{\tau_k}=\{f \in H(X) : d(f(x),x))<\epsilon$ for all $x \in X\}$, for some $\epsilon >0$.
For any subset $S$ of $X$ let $H(X|S)=\{ h \in H(X) : h$ is the identity outside $S\}$.

\begin{thm}\label{nomin}
Let $(X,d)$ be a compact metric space. Let $T$ be a non-empty open subset of $X$ containing no isolated points. Suppose $H(X|T)$ has the following two properties:

(A$_T$) for every $\epsilon>0$ and $x_1,y_1,x_2,y_2,\dots,x_n,y_n$ distinct points in $T$ such that $d(x_i,y_i)<\epsilon$ for $i=1,\dots n$,  there exists $h\in B_\epsilon^{\tau_k} \cap H(X|T)$ such that $h(x_i)=y_i$ for $i=1,\dots,n$;

(B$_T$) for every finite subset $F$ of $T$, the set $B_F^{\tau_p} \cap H(X|T)$ is highly transitive on $T \setminus F$.

Then if $\tau$ is a topological group topology on $H(X)$,  $\tau \subseteq \tau_p \cap \tau_k$ and $1 \in U \in \tau$ then $\{1\} \ne H(X|T) \subseteq U$. 
Hence, $H(X)$ does not have a minimum Hausdorff group topology.
\end{thm}
\begin{proof} Let $\tau$ be a topological group topology on $H(X)$. 
Take any $U$ in $\tau$ containing $1$ and any $h$ in $H(X|T)$. We will show $h\in U$. Clearly by either condition (A$_T$) or (B$_T$) we have many $1 \ne h \in H(X|T)$. Hence no $\tau$-open set separates $1$ from any such $h$,  so $\tau$ is not Hausdorff, and $H(X)$ does not have a minimum Hausdorff group topology. 

As $\tau$ is a group topology, we can find $U'$ in $\tau$ such that $U'$ is symmetric  and $1 \in U' \subseteq (U')^4 \subseteq U$.

 As $U'$ in $\tau$ and $\tau \subseteq \tau_p$, there is a finite subset $F$ of $X$ such that $B_F^{\tau_p} \subseteq U'$. Let $F'=F \cap T$ and enumerate $F'= \{x_1, \ldots, x_n\}$. As $U'$ in $\tau$ and $\tau \subseteq \tau_k$, there is an $\epsilon>0$ such that $B_\epsilon^{\tau_k} \subseteq U'$. We can suppose that $\epsilon < \min \{ d(x,x') : x \ne x', \, x,x' \in F \cup h^{-1}(F)\}$.

Pick $h_1$ in $H(X|T)$ such that: (1) $h_1 \in B_\epsilon^{\tau_k}$ and (2) for $i=1, \ldots , n$ the point $y_i=h_1(h^{-1}(x_i))$ is in $T$ but not in $F$. That $h_1$ exists follows from (A$_T$).

Pick $h_2$ in $H(X|T)$ such that: (1) $h_2 \in B_F^{\tau_p}$ and (2) for $i=1, \ldots , n$ the point $z_i=h_2(y_i)$ has $d(z_i,x_i)<\epsilon$. As $\{y_1, \ldots, y_n\} \cap F=\emptyset$, existence of $h_2$ is guaranteed by (B$_T$).

Pick $h_3$ in $H(X|T)$ such that: (1) $h_3 \in B_\epsilon^{\tau_k}$ and (2) $h(z_i)=x_i$ for $i=1, \ldots n$. As $d(z_i,x_i)<\epsilon$ for all $i$, the existence of $h_3$ is given by (A$_T$).

Let $h_4=h (h_3h_2h_1)^{-1}$. Then evidently $h=h_4h_3h_2h_1$. By construction $h_1,h_2,h_3$ are in $U'$. It remains to show that $h_4$ is in $B_F^{\tau_p}$ (which is contained in $U'$), for then $h \in (U')^4 \subseteq U$.

As $h, h_3, h_2$ and $h_1$ are in $H(X|T)$, so is $h_4$. Thus $h_4$ is in $B_F^{\tau_p}$ if and only if $h_4(x_i)=x_i$ for $i=1,\ldots, n$, and this occurs if and only if $h_4^{-1}(x_i)=x_i$ for $i=1, \ldots , n$. Fix an $i$. Then: $h_4^{-1}(x_i)=(h_3h_2h_1)h^{-1} (x_i) = (h_3h_2) (h_1 (h^{-1}(x_i))) = h_3 (h_2(y_i)) = h_3(z_i)=x_i$.
\end{proof}

In some cases we can apply the preceding theorem with $T=X$ and deduce:
\begin{cor}\label{nomin_MCH}
For the following spaces $X$ the only group topology contained in both the topology of pointwise convergence and the compact-open topology is the indiscrete topology. In particular, $H(X)$ does not have a minimum Hausdorff group topology. 

(a) Every compact manifold (without boundary) of dimension at least $2$,

(b) the Cantor set, and

(c) the Hilbert cube.
\end{cor}

Taking $T$ to be the hypothesized open $n$-cell, for $n \ge 2$ we deduce:
\begin{cor}
If $X$ is a compact metric space containing an open $n$-cell, for $n \ge 2$, then $H(X)$ does not have a minimum Hausdorff group topology.
\end{cor}

\section{Shrinking the Compact-Open Topology}

In this section we introduce a method of `shrinking' the compact-open topology around a closed subset. Under reasonable conditions this yields a topological group topology. Let $(X,d)$ be compact metric. Let $C$ be a closed subset of $X$. For $\epsilon >0$ define $C_\epsilon=C_\epsilon^d=\{x \in X : d(x,C)<\epsilon\}$.
Let $\tau_{k|C}$ be the collection of all unions of all translates of sets of the form: $B_\epsilon^{\tau_{k|C}} = \{ h \in H(X) : d(h(x),x)<\epsilon, d(h^{-1}(x),x)<\epsilon$ for all $x \in X\setminus C_\epsilon\}$. Note that if $d'$ is another metric on $X$ which is compatible with $\mathcal{T}_d$, then by compactness,  given $\epsilon>0$ there is a $\delta>0$ such that $C_\delta^d \subseteq C_\epsilon^{d'}$ and $C_\delta^{d'} \subseteq C_\epsilon^{d}$. So $\tau_{k|C}$ is independent of the choice of compatible metric.

To verify that we have indeed defined a topological group topology we apply the following well known result (see \cite[13G.6]{Wil} for example).
\begin{thm}\label{NbdChTG} Let $G$ be a group with identity $1$. 
Suppose $\mathscr{U}$ is a family of subsets of $G$ all containing $1$ satisfying the following conditions: (a) for each $U \in \mathscr{U}$, there exists $V \in \mathscr{U}$ with $V^2 \subseteq U$; 
(b) for each $U \in \mathscr{U}$, there exists $V \in \mathscr{U}$ with $V^{-1} \subseteq U$; 
(c) for each $U \in \mathscr{U}$ and $x \in \mathscr{U}$, there exists $V \in \mathscr{U}$ with $xV \subseteq U$; 
(d) for each $U \in \mathscr{U}$ and $x \in G$, there exists $V \in \mathscr{U}$ with $xVx^{-1} \subseteq U$; and
(e) for each $U, V \in \mathscr{U}$, there exists $W \in \mathscr{U}$ with $W \subseteq U \cap V$. 

Then $\{xU : x \in G, \ U \in \mathscr{U}\}$ is a basis for a topology $\tau$ on $G$ making $G$ into a topological group. The topology $\tau$ is $T_1$ (or equivalently, Hausdorff, or Tychonoff) if and only if $\bigcap \mathscr{U}=\{1\}$.
\end{thm}

\begin{thm} \label{thm:taukC} Let $(X,d)$ be compact metric, $C$ be a closed subset of $X$ and define $\tau_{k|C}$ as above. Then the following hold:

(1) $(H(X),\tau_{k|C})$ is a group topology \emph{provided} $C$ is $H(X)$-invariant ($h(C)=C$ for all $h$ in $H(X)$),

(2) $\tau_{k|C}$ is $T_0$ if $C$ is closed nowhere dense, and

(3) $\tau_{k|C} \subseteq \tau_k$ with equality if and only if for every $\epsilon > 0$ there is an $0<\delta \le \epsilon$ such that if $h$ in $H(X)$ satisfies: $d(h(x),x) < \delta$ and $d(h^{-1}(x),x) < \delta$ for all $x \notin C_\delta$ then $d(h(x),x) < \epsilon$ and $d(h^{-1}(x),x) < \epsilon$ for all $x \in C_\delta$.
\end{thm}

\begin{proof} \ 

\noindent {\bf For (1):} We verify conditions (a) through (e) of the neighborhood characterization of a topological group topology (Theorem~\ref{NbdChTG}).

(a) We need  to show: for every $\epsilon >0$ there is a $\delta>0$ such that $B_\delta^2 \subseteq B_\epsilon$. 

To this end fix $\epsilon >0$ and set $\delta = \epsilon/2$. Take  any $g_1, g_2 \in B_\delta$. Fix $x \not \in C_\epsilon$, then $x \not \in C_\delta$. Since $d(g_1(x), x) < \delta$ we see $g_1(x) \not \in C_\delta$,  so $d(g_2 g_1(x), g_1(x)) < \delta $ and then  $d(g_2 g_1(x), x) < 2 \delta < \epsilon$. 
Similarly $d(g_1^{-1}g_2^{-1}(x), x) < \epsilon$. 
So for all $g_1, g_2 \in B_\delta, g_2 g_1, \in B_\epsilon$, in other words  $B_\delta^2 \subseteq B_\epsilon$, as required.

(b) We need to show: for every $\epsilon>0$ there is a $\delta>0$ such that $B_\delta^{-1} \subseteq B_\epsilon$. 
But by definition of $B_\epsilon$, we have $B_\epsilon = B_\epsilon^{-1}$, so we can take $\delta=\epsilon$.

(c) We need to show: for all $\epsilon >0$ and $f \in B_\epsilon$, there is a $\delta>0$ such that $fB_\delta \subseteq B_\epsilon$. 

Fix then $f \in B_\epsilon$, and note $f^{-1} \in B_\epsilon$. 
By compactness of $X$, \[\epsilon_1=\max  \{d(f(x), x), d(f^{-1}(x), x): x \not \in C_\epsilon\}  < \epsilon.\] 
Let $\epsilon_2 = \epsilon - \epsilon_1$. Since 
$f$ and $f^{-1}$ are uniformly continuous on $X \setminus C_\epsilon$, there is a $\delta' >0$ such that for every $x, y \in X$ if $d(x, y) < \delta'$ then $d(f(x), f(y)) < \epsilon_2$ and  $d(f^{-1}(x), f^{-1}(y)) < \epsilon_2$. 
Now let $\delta = \min \{ \delta', \epsilon_2 \}$. 
Take any $x \not \in C_\epsilon$. Then $x \not \in C_\delta$. Take any $g \in B_\delta$. Then $g^{-1} \in B_\delta$. 
So if $d(g(x), x) < \delta$ then $d(fg(x), f(x)) < \epsilon_2$, and  hence 
$d(fg(x), x) \le d(fg(x), f(x)) + d(f(x), x) \le \epsilon_2 + \epsilon_1 = \epsilon$. 
On the other hand, $d(f^{-1}(x), x) \le \epsilon_1$. 
Since $x \not \in C_\epsilon$, we know $f^{-1}(x) \not \in C_\delta$,   and so $d(g^{-1}f^{-1}(x), f^{-1}(x)) < \delta \le \epsilon_2$. 

Thus $d((fg)^{-1}(x), x) \le d(g^{-1}f^{-1}(x), f^{-1}(x)) + d(f^{-1}(x), x) < \epsilon_2 + \epsilon_1 =\epsilon$, and $fg \in B_\epsilon$ for all $g \in B_\delta$, as required.

(d) We need  to show: for every $\epsilon>0$, and $f \in H(X)$ there is a $\delta>0$ such that $f B_\delta f^{-1} \subseteq B_\epsilon$. 

Fix  $f \in H(X)$. We make two observations: 
(1) as $C_\epsilon$ is open in $X$ so $f(C_\epsilon)$ is open in $X$, and since $C$ is $H(X)$-invariant, there is a $\delta_1$ such that $C_{\delta_1} \subseteq f(C_\epsilon)$; and 
(2) as $f^{-1}$ is uniform continuous there is a $\delta_2$ such that for all $x, y \in X$ if  $d(x, y) < \delta_2$ then we have $d(f^{-1}(x), f^{-1}(y)) < \epsilon$. 

Let $\delta = \min \{ \delta_1, \delta_2 \}$. 
Observation (1) implies $x \not \in C_\epsilon, f(x) \not \in C_\delta$. 
So for all $g \in B_\delta$, we have $d(gf^{-1}(x), f^{-1}(x)) < \delta$ and $d(g^{-1}f^{-1}(x), f^{-1}(x)) < \delta$. 
Observation (2) implies $d(fg^{-1}f^{-1}(x), ff^{-1}(x)) < \epsilon$ and $d(fgf^{-1}(x), x) < \epsilon$. 
Hence $fB_\delta f^{-1} \subseteq B_\epsilon$, as required.

(e) We need to show: for all $\epsilon_1, \epsilon_2$ there is a $\delta > 0$ such that $B_\delta \subseteq B_{\epsilon_1} \intersect B_{\epsilon_2}$. 
But after setting $\delta = \min \{ \epsilon_1, \epsilon_2 \}$ we are done.

\medskip

\noindent {\bf For (2):}  Suppose $C$ is closed nowhere dense. We show $\tau_{k|C}$ is $T_0$, by verifying that $\{1\} = \bigcap \{B_\epsilon: \epsilon > 0 \}$. 

Clearly $\{1\} \subseteq \bigcap \{B_\epsilon: \epsilon > 0 \}$. 
If, for a contradiction, there is a $g$ in the intersection such that $g \neq 1$, then there must be an $x \in X$ such that $d(g(x), x) = \epsilon > 0$. 
If $x \in X \setminus C$, then $d(x, C) = \epsilon' > 0$. 
Let $\delta = \min \{ \epsilon', \epsilon \}$, then $d(g(x), x) > \delta/2$, so $g \not \in B_{\delta/2}$. 
Contradiction! 
It follows that $g(x) \neq x$, for some $x \in C$ and $g(y) = y$ for all $y \not \in C$. 
Since $g$ is uniformly continuous, there is a $\delta$ such that if $d(x, y) < \delta'$ then $d(g(x), g(y)) < \epsilon/3$. 
Let $\delta = \min \{ \delta', \epsilon/3 \}$. 
As $C$ is closed and nowhere dense, $C$ does not contain any open ball in $X$, in particular $\mathcal{B}_\delta(x) := \{t \in X: d(t,x)<\delta \} \not\subseteq C$. 
So there is a $y \not \in C$ such that $d(x, y) < \delta$.  
Then $\epsilon  = d(f(x), x) \le d(x, y) + d(y, f(y)) + d(f(y), f(x))   < \epsilon/3 + 0 + \epsilon/3 < \epsilon$. 
Contradiction again, and Claim~(2) is established. 

\medskip

\noindent {\bf For (3):} Clearly $\tau_{k|C} \subseteq \tau_k$. We verify that  equality holds if and only if $\epsilon > 0$ there is a $0<\delta \le \epsilon$ such that if $h$ in $H(X)$ satisfies: $d(h(x),x) < \delta$ and $d(h^{-1}(x),x) < \delta$ for all $x \notin C_\delta$ then $d(h(x),x) < \epsilon$ and $d(h^{-1}(x),x) < \epsilon$ for all $x \in C_\delta$ ($\ast$). 

\smallskip

($\Rightarrow$) Suppose equality holds, and in particular $\tau_k \subseteq \tau_{k|C}$. Then given any $\epsilon>0$ we know  $B_\epsilon^{\tau_k}$ is open in $\tau_{k|C}$. 
So there is a $\delta>0$  such that $B_\delta^{\tau_{k|C}} \subseteq B_\epsilon^{\tau_k}$. Of course we may assume $\delta \le \epsilon$. We show $\delta$ satisfies  ($\ast$).

So take any $h$ in $H(X)$ such that $d(h(x),x) < \delta$ and $d(h^{-1}(x),x) < \delta$ for all $x \notin C_\delta$. Then  $h$ is in $B_\delta^{\tau_{k|C}}$. So $h$ is in $B_{\epsilon}^{\tau_k}$, which means $d(h(x),x) < \epsilon$ and $d(h^{-1}(x),x) < \epsilon$ for all $x \in X$, and hence certainly all $x$ in $C_\delta$.

\smallskip

($\Leftarrow$) Now suppose we have ($\ast$). Take any basic $B_\epsilon^{\tau_k}$. We need to show it is in $\tau_{k|C}$ -- in other words, we need to find a $0<\delta$ so that $B_\delta^{\tau_{K|C}} \subseteq B_\epsilon^{\tau_k}$, in other words: if $h$ is in $H(X)$ and $d(h(x),x) < \delta$ and $d(h^{-1}(x),x) < \delta$ for all $x \notin C_\delta$ then $d(h(x),x) < \epsilon$ and $d(h^{-1}(x),x) < \epsilon$ for all $x \in X$.

However, let $\delta$ be as given by ($\ast$).   If $h \in H(X)$ and  $d(h(x),x) < \delta$ and $d(h^{-1}(x),x) < \delta$ for all $x \notin C_{\delta}$ we have $d(h(x),x) < \epsilon$ and $d(h^{-1}(x),x) < \epsilon$ for all $x \in C_\delta$. Since $\delta \le \epsilon$ we also have $d(h(x),x) < \epsilon$ and $d(h^{-1}(x),x) < \epsilon$ for all $x \notin C_\delta$. And so the  relevant inequalities hold for all $x$ in $X$ -- as needed.
\end{proof}

 Let $M$ be a compact manifold with non-empty boundary, $C$. Applying the above we get that $\tau_{k|C}$ is the minimum  Hausdorff topological group topology on $H(M)$. It is straightforward to check that condition 3) fails, so $\tau_k$ is not minimal.
\begin{cor} The compact-open topology on the homeomorphism group of a compact manifold with non-trivial boundary is not minimal. 
\end{cor}
In particular, the compact-open topology on the homeomorphism group of the Mobius band is not minimal, answering Question~4.28 of \cite{DM}.

\section{Homeomorphism Groups with Minimum}

Let $X$ be a space. Define $O_X$ (`the one-manifold part of $X$') to be $O_X = \{x\in X: \ x$  has a neighborhood homeomorphic to   $(0,1)\}$. Define  $S_X$ to be all points with a clopen neighborhood homeomorphic to the circle, $S^1$, and $I_X=O_X \setminus S_X$. Note that each point of $I_X$ has a neighborhood homeomorphic to $(0,1)$ which is clopen in $O_X$. Define $C_X=X  \setminus O_X$.
Suppose $X$ is compact and metrizable. Then observe: 
(a) $I_X$ is a countable (possibly empty) disjoint sum of copies of $(0,1)$, and $S_X$ is a countable (possibly empty) disjoint sum of circles, and 
(b) $O_X$ is open, and the disjoint sum of $I_X$ and $S_X$.

For any homeomorphism $h$ of a space $X$, define  $\move (h) = \{ x \in X: h(x) \neq x \}$.

The key result on the existence of a minimum Hausdorff group topology  on certain homeomorphism groups is the following. 
\begin{thm}\label{thm:taukc_min}
If $X$ is a compact metrizable space  such that $O_X$ is dense in $X$, then $\tau_{k|C_X}$ is the minimum Hausdorff topological group topology on $H(X)$. 
\end{thm}

However, to deal with automorphism groups in the final section we prove a strengthening of Theorem~\ref{thm:taukc_min} to include certain subgroups of the homeomorphism group. 

\begin{thm}\label{thm:taukc_min_GEN}
Let $X$ be a compact metrizable space  such that $O_X$ is dense in $X$. Let $G$ be a subgroup of $H(X)$ such that for every open subset $U$ of $X$ homeomorphic to $(0,1)$  there is a non-trivial $g$ in $G$ with $\move(g) \subseteq U$.


Then the topology $\tau_{k|C_X}$ on $H(X)$ restricted to $G$ is the  minimum Hausdorff group topology on $G$.
\end{thm}

\begin{proof}
Let $d$ be a compatible metric on $X$. Let $C = C_X$ and recall $C$ is the set of points in $X$ with no neighborhood homeomorphic to $(0, 1)$. Clearly $C$ is $H(X)$-invariant. As $O_X$ is open and dense by hypothesis, $C$ is closed and nowhere dense. Hence by Theorem~\ref{thm:taukC}, $\tau_{k|C}$ is a Hausdorff topological group topology on $H(X)$, contained in the compact-open topology $\tau_k$. We show its restriction to $G$ is the minimum Hausdorff topological group topology on $G$.

We denote the closed unit interval, $[0,1]$, by $I$ and write $S^1$ for the unit circle $\{(\cos (2\pi \theta), \sin (2\pi \theta)) : \theta \in [0,1)\}$. Given $a < b$ we write $[a,b]$ for the subset of $S^1$ given by $\{ (\cos (2\pi \theta), \sin (2\pi \theta)) : \theta \in [a,b]\}$. It will always be clear from context if `$[a,b]$' is an interval in $\R$ or a subset of $S^1$.

Fix $\epsilon > 0$, let $C_\epsilon = \{x \in X: d(x, C) < \epsilon\}$. We need to show $B=B_\epsilon^{\tau_{k|C}} \cap G=\{ h \in G : d(h(x),x) < \epsilon \text{ and } d(h^{-1}(x),x) < \epsilon \text{ for all } x \notin C_\epsilon\}$ is open in every Hausdorff group topology on $G$.

By compactness of $X$, $X \setminus C_\epsilon \subseteq  \bigoplus_{n=1}^N I_n \oplus \bigoplus_{m=1}^M S_{N+m} \subseteq X \setminus C$, where each $I_n$ is homeomorphic to a closed interval and each $S_{N+m}$ is homeomorphic to a circle. Each $I_n$ is contained in a $J_n$ homeomorphic to $\R$, and we  fix a homeomorphism $f_n$ of $J_n$ with $\R$ so that $f_n$ maps $I_n$ to $I \subseteq \R$. For each $S_{N+m}$ fix a homeomorphism $f_{N+m}$ of $S_{N+m}$ to $S^1$.

Take any open subset $U$ of $X$ which is homeomorphic to $(0,1)$. By hypothesis there is a $p$ in $G$ such that $p \ne 1$ and $\move(p) \subseteq U$. As $p$ is non-trivial we can find an open subset $V$ of $U$ which is homeomorphic to an open sub-interval of $U$ such that $p(V) \cap V = \emptyset$. Applying the hypothesis to $V$ there is a $q$ in $G$ such $\move(q) \subseteq V$ and $q(x) \ne x$ for some $x$ in $V$. Now we see that $p$ and $q$ do not commute: since $x \in V$, $p(x) \notin V$, so $q(p(x))=p(x)$, but $q(x) \ne x$ forces $p(q(x)) \ne p(x)$. 
Thus for any $a < b$ and $k \le N+M$ we can find non-commuting $p_k^{(a,b)}$ and $q_k^{(a,b)}$ in $G$ such that $f_k(\move(p_k^{(a,b)}))$ and $f_k(\move(q_k^{(a,b)}))$ are subsets of $[a,b]$. Fix $k \le N+M$ and define $T(a,b,k) = \{g \in G : g p_k^{(a,b)} g^{-1} \text{ and } q_k^{(a,b)}$ do not commute and $g^{-1} p_k^{(a,b)} g \text{ and } q_k^{(a,b)}$  do not commute$\}$.

\begin{claim}
In any Hausdorff group topology on $G$ the set $T(a, b,k)$ is open and contains $1$.
\end{claim}

 For notational  convenience write $p$ for $p_k^{(a,b)}$ and $q$ for $q_k^{(a,b)}$.
That $1$ is in $T(a,b,k)$ is immediate since  $1p1^{-1} = p = 1^{-1}p1$, and $p=p_k^{(a,b)}$ does not commute with $q=q_k^{(a,b)}$.

Now consider the  maps
$\varphi, \psi: G  \rightarrow G$ defined by $\varphi(g)= gpg^{-1} q (gpg^{-1})^{-1} q^{-1}$ and $\psi(g)=g^{-1}pg q (g^{-1}pg)^{-1} q^{-1}$.
These maps are continuous for every topological group topology on $G$.
Note that $\varphi(g) = 1$ if and only if $gpg^{-1}$ and $q$ commute.
Similarly, $\psi(g) = 1$ if and only if $g^{-1}pg$ and $q$ commute.
Since $\{1\}$ is closed in any Hausdorff topology on $G$, the set $V = G \setminus \{1\}$ is open, and then the set $\varphi^{-1}(V) \intersect \psi^{-1}(V) = T(a, b,k)$ is open in any Hausdorff topological group topology, as required.

\begin{claim}\label{claimT}
If $g \in T(a,b,k)$, then  there exists $x, y$ in $f_k^{-1}[a,b]$ such that $g(x)$ and $g^{-1}(y)$ are in $f_k^{-1}[a,b]$.
\end{claim}

Let $U_k=f_k^{-1} [a,b]$.
We prove by contradiction that there is an $x$ in $U_k$ such that $g(x)$ is in $U_k$. By symmetry in the definition of $T(a,b,k)$, as $g$ is in $T(a,b,k)$, so is $g^{-1}$, and  the existence of the required $y$ follows.

For notational simplicity, write $p$ for $p_k^{(a,b)}$ and $q$ for $q_k^{(a,b)}$. For concreteness let us suppose $k=N+m$, so $U_k$ is a subset of $S_k$. (The case when $k \le N$ is identical except each occurrence below of `$S_k$' should be replaced with `$J_k$'.)
For the desired contradiction suppose that for all $x$ in $U_k$ we have $f_k(g(x)) \notin [a,b]$.

Note that if $z \in g(S_{k} \setminus U_{k})$, then $g^{-1}(z) \in S_{k} \setminus U_{k}$.
This implies $pg^{-1}(z) = g^{-1}(z)$, because $f_{k}(\move(g)) \subseteq [a, b]$, and $gpg^{-1}(z) = gg^{-1}(z) = z$.
Therefore $z \not \in \move(gpg^{-1})$, so $\move(gpg^{-1}) \subseteq g(U_{k})$.

Since $g(U_{k}) \subseteq X \setminus U_{k}$, we have $\move(gpg^{-1}) \subseteq X \setminus U_k$.
Then $\move(gpg^{-1})$ and $\move(q)$ are disjoint.

Fix $x \in S_{k}$.
 If $x \in U_{k}$, then $q(gpg^{-1})(x) = q(x) \subseteq U_{k}$, so $(gpg^{-1})q(x) = q(x)$.
Thus $q$ and $gpg^{-1}$ commute  at $x$.
 Similarly, if $x \in S_{k} \setminus U_{k}$, then $q$ and $gpg^{-1}$ commute at $x$.

Hence $gpg^{-1}$ and $q$ commute on $S_{k}$. By choice of $p=p_k^{(a,b)}$ and $q=q_k^{(a,b)}$, $gpg^{-1}$ and $q$ certainly commute outside $S_{k}$ (where they are the identity).
Since $g$ is in $T(a,b,k)$, we have our desired contradiction.

\medskip

We now show that some finite intersection of $T(a,b,k)$'s is contained in the given basic $B$. Since all $T(a,b,k)$'s are open neighborhoods of the identity in any Hausdorff group topology on $G$, this completes the proof that $\tau_{k|C_X}$ restricted to $G$ is the minimum Hausdorff topological group topology on $G$.

To this end, pick $t \in \N$ such that for all $k \le N+M$ and $i=-2,0, \ldots , t$ the $d$-diameter of $f_k^{-1} [i/t,(i+1)/t]$ is $< \epsilon/3$.
Define
\[T= \bigcap \left\{T(i/t,(i+1)/t,k) : k \le N+M, \ i=-2,0,\ldots , t \right\}.\]

\begin{claim}
If $h \notin B$ then $h \notin T$.
\end{claim}

Take any $h$ in $G$ which is not in $B$. So there is an $x \in X \setminus C_\epsilon$ such that $d(h(x), x) \ge \epsilon$ or $d(h^{-1}(x), x) \ge \epsilon$. The sets $T(a,b,k)$ in the definition of $T$ are symmetric -- if $h$ is in $T(a,b,k)$ then so is $h^{-1}$ -- so we may assume, without loss of generality, that $d(h(x),x)) \ge \epsilon$.
Then  $x \in I_n$ for some $n$, in which case set $k=n$ and $A=I_k$, or $x \in S_{N+m}$ for some $m$, in which case set $k=N+m$ and $A=S_k$.

Consider an interval $K=[i/t, (i+1)/t]$ in the circle $S^1$. Let $K_- = [i_-/t, (i_- +1)/t]$ and $K_+= [i_+/t, (i_+ +1)/t]$, where $i_- = (i-1) \pmod t$ and $i_+=(i+1) \pmod t$. So $K_-$ is the interval `preceding' $K$ and $K_+$ is the interval `succeeding' $K$ in the natural cyclic order. Similarly, consider an interval $K=[i/t, (i+1)/t]$ in the closed unit interval $I$. Let $K_- = [i_-/t, (i_- +1)/t]$ and $K_+= [i_+/t, (i_+ +1)/t]$, where $i_- = i-1$ and $i_+=i+1$.

There is a unique interval $K=[i/t,(i+1)/t]$ such that $f_k(x)$ is in $K$ but not in $K_-$. Note that $-1 \le i \le t-1$.
We will show $h$ is not in (at least) one of $T(i_-/t, (i_-+1)/t,k)$, $T(i/t, (i+1)/t,k)$, or $T(i_+/t, (i_++1)/t,k)$, and hence is not in $T$, as desired.

Otherwise, from the preceding Claim, we can pick $x_{-1}$ in $f_k^{-1} K_-$ such that $h(x_{-1})$ is also in $f_k^{-1} K_-$, $x_0$ in $f_k^{-1} K$ such that $h(x_0)$ is  in $f_k^{-1} K$,   and $x_1 \in f_k^{-1} K_+$ such that $h(x_1) \in f_k^{-1} K_+$. Observe that it follows that $h$ maps $J_k$ (respectively, $S_k$) homeomorphically to $J_k$ (respectively, $S_k$) if $k \le N$ (respectively, $k>N$).

Indeed as $f_k(x_{-1}) < f_k(x_0) < f_k(x_{1})$ (either in the cyclic order on $S^1$, or standard order on $\R$) $h$ maps $f_k^{-1} [x_{-1},x_{1}]$ into $f_k^{-1} (K_- \cup K \cup K_+)$. In particular $h(x)$ is in $f_k^{-1}(K_- \cup K \cup K_+)$. But as the $d$-diameter of each of $f_k^{-1} K_-$, $f_k^{-1} K$ and $f_k^{-1} K_+$ is no more than $\epsilon/3$, this means $d(h(x),x) \le 2(\epsilon/3)$, contradicting $d(h(x),x) \ge \epsilon$.
\end{proof}

\subsection{Conditions on $X$ ensuring  $\tau_{k|C} \ne \tau_k$}

Here we give sufficient conditions for the minimum topology to be strictly smaller than the compact-open topology. Note that this implies that the compact-open topology is not minimal.

\begin{prop}\label{suff_tkcnetk}
Let $X$ be compact metrizable, and suppose $O_X$ is dense in $X$. Let $C=C_X$. In the following cases $\tau_{k|C} \ne \tau_k$, and the compact-open topology on $H(X)$ is not minimal:

(a) $C$ contains at least two points which are the limit of clopen circles ($|\cl{S_X} \cap C| \ge 2$), or

(b) there is a component $I$ of $I_X$ whose closure meets $C$ in at least three points ($|\cl{I} \cap C| \ge 3$).
\end{prop}
\begin{proof} \ Fix a compatible metric $d$ for $X$. In both cases we verify that the condition for equality of $\tau_{k|C}$ and $\tau_k$ in Theorem~\ref{thm:taukC}~(3) fails.

\noindent {\bf For (a):} Let $x_1$ and $x_2$ be distinct points in $C$ which are the limit of circles in $S_X$. Let $\epsilon= d(x_1,x_2)/2$. Take any $0<\delta \le \epsilon$. Then we can find circles $S_1$ and $S_2$ in $S_X$ so that $S_i \subseteq B_d(x_i,\delta/4)$ for $i=1,2$. Define $h$ to be a homeomorphism of $X$ which is the identity outside $S_1 \cup S_2$ and which switches $S_1$ and $S_2$. Then $h$ is the identity outside $C_\delta$ but moves points of $C_\delta$ (namely all those in $S_1 \cup S_2$) at least $d(x_1,x_2) - (\delta/4 + \delta/4) \ge d(x_1,x_2)/2 =\epsilon$. 

\noindent {\bf For (b):} Fix the open interval $I$ in $I_X$ such that $|\cl{I} \cap C| \ge 3$. The set $\cl{I} \setminus I$, which is the remainder of a compatification of an open interval,  has either one or two components. By hypothesis we can pick two distinct points of $C$, say $x_1$ and $x_2$, which are in the same component. Then $x_1$ and $x_2$ are in the closure of a ray, $J$, (homeomorphic to $[1/2,1)$) contained in $I$. Let $\epsilon=d(x_1,x_2)/2$.  Take any $0<\delta \le \epsilon$. Without loss of generality we can assume the ray $J$ is contained in $C_\delta$.
We  can pick $y_1, y_1'$ and $y_2,y_2'$ in the ray $J$ so that: $y_1$ precedes $y_1'$ in $J$, $y_1'$ precedes $y_2'$, $y_2'$ precedes $y_2$ in $J$, $d(x_i,y_i) < \delta/4$ and $d(x_i,y_i') < \delta/4$ for $i=1,2$. Pick a homeomorphism $h$ of $X$ which is the identity outside the closed subinterval of $J$ between $y_1$ and $y_2$, but which moves $y_1'$ to $y_2'$. Then as in case (a), $h$ is the identity outside $C_\delta$ but moves a point (namely, $y_1'$) of $C_\delta$ at least $d(x_1,x_2)/2=\epsilon$. 
\end{proof}

\begin{eg} \ 

(a) The  homeomorphism group of the disjoint sum of two convergent sequences of circles, with the two limit points, ($X=\bigcup_n S_n(0,0) \cup \bigcup_n S_n(2,0) \cup \{(0,0),(2,0)\}$, where $S_n(\mathbf{x})$ is the circle of radius $1/n$ centered at $\mathbf{x}$), has a minimum group topology, but the compact-open topology is not minimal.

(b) The homeomorphism group of the topologist's sine curve has a minimum group topology, but the compact-open topology is not minimal.

(c) For every compact metric space $K$ there is a compact metric $X=X_K$ such that $O_X$ is dense and the countably infinite disjoint sum of circles, and $C_X$ is homeomorphic to $K$. This space $X_K$ has a minimum group topology, but the compact-open topology is not minimal.

(d) For every non-zero dimensional compact metric space $K$ there is a compact metric $X=X_K$ such that $O_X$ is dense and the countably infinite disjoint sum of open intervals, and $C_X$ is homeomorphic to $K$. This space $X_K$ has a minimum group topology, but the compact-open topology is not minimal.
\end{eg}
\begin{proof} Examples (a) and (b) are easy exemplars of cases (a) and (b) of the above proposition. We sketch (c) and (d).

For (c) recall that every compact metric space is the remainder of a compactification of $\N$. It is clear we can replace each $n$ in  $\N$ with a circle. This gives $X_K$. That $\tau_{k|C} \ne \tau_k$ is immediate from (a) of the preceding proposition.

For (d) first note that every compact metric space, $K$, is the remainder of a compactification, $Z_K$, of a countably infinite disjoint sum of open intervals (see Lemma~\ref{lotsExs} below). Our hypothesis is that $K$ is not zero-dimensional. So it is not totally disconnected, and contains a non-trivial component $K'$. Now we can add an open interval, $I$, to $Z_K$, to get a new compact metric space $X_K$, so that $\cl{I} \setminus I = K'$. Then $X_K$  has the required topological properties, and has $\tau_{k|C} \ne \tau_k$ by case (b) of the preceding proposition.

\end{proof}

\subsection{Conditions on $X$ ensuring  $\tau_{k|C}=\tau_k$}

Here we give sufficient conditions for the minimum topology to coincide with the compact-open topology. Note that this implies that the compact-open topology is  minimal. The relevant spaces are similar to the class of (compact, metrizable) \emph{graph-like} spaces recently introduced by Thomassen and Vella \cite{ThVe}. Given a topological space $X$, an edge of $X$ is an open subset homeomorphic to $(0,1)$, whose closure is a simple arc. Then $X$ is  graph-like if there is a collection $E$ of pairwise disjoint edges of $X$ such that $X \setminus E$ is zero-dimensional.

\begin{prop}\label{suff_tkceqtk}
Let $X$ be compact metrizable, and suppose $O_X$ is dense in $X$. Let $C=C_X$. In the following cases  $\tau_{k|C}=\tau_k$, and the compact-open topology on $H(X)$ is minimal:

(a)   $C$ is zero-dimensional, $X \setminus S_X$ has only finitely many components, and $S_X$ has at most one limit point in $C$, or

(b) $C$ is a convergent sequence and $S_X$ has at most one limit point in $C$.
\end{prop}

\begin{proof} Fix a compatible metric $d$ for $X$.

\noindent {\bf For (a):} We will prove this in three steps. In all cases $C=C_X$ is zero-dimensional. 

Let us start by assuming that $X$ is connected and $S_X$ is empty. We claim that $X$ is path connected. This is not difficult to prove directly. Alternatively we can argue as follows. Since $O_X$ is dense but $S_X$ is empty, we have that $I_X$ is dense. Let $I_X^1$ be all components (open intervals) in $I_X$ whose closure is a circle and $I_X^2$ be all open intervals in $I_X$ whose closure is an arc. As $C$ is zero-dimensional, every open interval in $I_X$ has closure either a circle or an arc. As $X$ is connected it is clear that $X'=X \setminus I_X^1$ is also connected.
Then $X'$ is a graph-like metric continuum, so \cite[2.1]{ThVe} is locally connected. Then we see that $X'$ is a Peano continuum, and so path connected. Reattaching all of the open intervals in $I_X^1$ (and their unique limit point in $C$) preserves path connectedness, so $X$ is path connected.

Now we show that $\tau_{k|C}=\tau_k$ provided $S_X=\emptyset$ and $X$ has finitely many components. Note that from above each component of $X$ is path connected. We verify that the condition in Theorem~\ref{thm:taukC}~(3) for equality of $\tau_{k|C}$ and $\tau_k$ is satisfied.

Fix $\epsilon > 0$. 
Since $C$ is zero dimensional and $X$ has finitely many components, we can partition $C$ into finitely many pieces, say $C_1, \ldots, C_n$, each of $d$-diameter strictly less that $\epsilon/3$, ensuring that each component of $X$ contains at least two pieces of the partition.  Pick positive $\delta$ smaller than $\epsilon/3$ and $\frac{1}{3} \min \{d(C_i,C_j): i \ne j \}$. Note that $C_\delta$ is the disjoint union of $(C_i)_\delta = \{ x \in X : d(C_i,x) < \delta\}$ for $i=1, \ldots ,n$. 

Take any $h$ in $H(X)$ such that $d(h(x),x)<\delta$ and $d(h^{-1}(x),x)<\delta$ for all $x$ not in $C_\delta$. Suppose, for a contradiction, that for some $x$ in $C_\delta$ either $d(h(x),x) \ge \epsilon$ or $d(h^{-1}(x),x) \ge \epsilon$. We can assume, without loss of generality, that in fact $d(h(x),x) \ge \epsilon$.

Suppose $h(x)$ is not in $C_\delta$. Then let $y=h(x)$. Now $y \notin C_\delta$ but $d(h^{-1}(y),y)=d(x,h(x)) \ge \epsilon > \delta$, contradicting the hypothesis on $h$.
So $h(x)$ is in $C_\delta$. Fix $i$ such that $x$ is in $(C_i)_\delta$. Fix $j$ such that $h(x)$ is in $(C_j)_\delta$. Since the diameter of $C_i$ is $<\epsilon/3$ and $\delta < \epsilon/3$, the diameter of $(C_i)_\delta$ is $<\epsilon$. Since $d(h(x),x) \ge \epsilon$ we see $i \ne j$.

As each component of $X$ is path connected and contains at least two pieces of the partition of $C$, there is an arc $A$ from $x$ to some point not in $(C_i)_\delta$ (indeed to a point in some $C_k$ where $k \ne i$). As we travel along the arc $A$ from $x$ there is a (unique) first point, $x'$, where $A$ exits $(C_i)\delta$. Then $x'$ is not in $C_\delta$, so $d(h(x'),x')<\delta$, but $d(x',C_i) \le \delta$, and hence $h(x')$ cannot be in $(C_j)_\delta$. 

Now we see that $h(A)$ is an arc starting at $y=h(x)$ which exits $(C_j)_\delta$, say for the first time at $y'$. Then $y'$ is not in $C_\delta$, so $d(h^{-1}(y'),y') < \delta$, but $d(y',(C_j)_\delta) \le \delta$, and hence $x''=h^{-1}(y')$ cannot be in $(C_i)_\delta$. But $x''$ is in $(C_i)_\delta$, since $x''$ is on $A$ before it exits $(C_i)_\delta$ for the first time at $x'$. Contradiction!

To complete the proof of (a), suppose $S_X$ has at most one limit point in $C$ and $X \setminus S_X$ has finitely many components. Note that every $h$ in $H(X)$ maps each circle in $S_X$ to another circle in $S_X$, and if $S_X$ has a unique limit point then it is a fixed point of $h$. Take any $\epsilon>0$. If $S_X$ has no limit in $C$ (so it is a finite union of circles) then pick $\delta_0>0$ so that $C_{\delta_0}$ is disjoint from $S_X$. Otherwise, let $x_0$ be the unique limit in $C$ of $S_X$. In this case pick $\delta_0>0$ so that every circle in $S_X$ meeting $B_\delta(x_0)$ has diameter $<\epsilon/3$. Now apply the previous step to $X'=X \setminus S_X$ and get a relevant $\delta_1>0$. Let $\delta=\min (\delta_0,\delta)_1)$ . It is straightforward to check that if $h$ in $H(X)$ is such that $d(h(x),x)<\delta$ and $d(h^{-1}(x),x) < \delta$ for all $x \in X \setminus C$, then $d(h(x),x)$ and $d(h^{-1}(x),x)$ are both $<\epsilon$ for all $x$ in $C_\delta$.

\noindent {\bf For (b):}  We assume $S_X = \emptyset$. Otherwise we can deal with $S_X$ as we did at the end of case (a).
By hypothesis $C$ is a convergent sequence, say with unique limit point $c$.

Fix $\epsilon > 0$. Then (replacing $\epsilon$ with something smaller, if necessary) we can find $N$ and write $C=\{c_n: n \in \N\}$ so that  $c_n$ is not in $\cl{B_{\epsilon/3}(c)}$ for $n \le N$ but $c_n \in B_{\epsilon/3}(c)$ for $n>N$.
Pick $\delta >0$ satisfying the following conditions: (i) $\delta < \epsilon/3$, (ii) $B_{2\delta}(c_i) \cap B_{2\delta}(c_j) = \emptyset$ for distinct $i,j \le N$, (iii) $B_{2\delta}(c_i) \cap \cl{B_{\epsilon/3}(c)} =\emptyset$ for all $i \le N$, (iv) for $i \le N$ if $I$ is a component of $I_X$ (so, an open interval) such that $I \cap B_{2\delta}(c_i)\ne \emptyset$ then $c_i \in \cl{I}$, and (v) for $i \le N$ if $I$ is a component of $I_X$ (so, an open interval) with two endpoints, one of which is $c_i$ then $B_\delta(c_i) \cap I$ is a subinterval of $I$. (Note for (iv) and (v) there are only finitely many open intervals $I$ to deal with for each $i\le N$.)

Fix $h \in H(X)$ such that $d(h(x),x)$ and $d(h^{-1}(x),x)$ are both $<\delta$ for all $x$ not in $C_\delta$. Suppose, for a contradiction, that for some $x$ in $C_\delta$ either $d(h(x),x) \ge \epsilon$ or $d(h^{-1}(x),x) \ge \epsilon$. We can assume, without loss of generality, that in fact $d(h(x),x) \ge \epsilon$.

As $C$ is $h$-invariant, $h$ must take $c$ to $h(c)$ and any $c_i$ to some $c_j$. By choice of $\delta$ and restriction on $h$ we see that in fact $h$ is the identity on $\{c\} \cup \{c_1,\ldots, c_N\}$. Then for $i>N$ we must have that $h(c_i)=c_j$ where $j>N$. It follows (as all $d(c_i,c_j)<\epsilon$ for all $i,j >N$) that $x$ cannot be in $C$. Further, $y=h(x)$ is in $C_\delta$, for if $y$ is not in $C_\delta$  then $d(h^{-1}(y),y)=d(x,h(x)) \ge \epsilon > \delta$, contradicting the hypothesis on $h$.

Thus $x$ is in an open interval, $I$, a component of $I_X$. Let $\{c_i,c_j\}=\cl{I}\setminus I$ be the endpoint(s) of this interval. If both $i,j >N$ then $h(x)$ is not in any $B_\delta(c_k)$ for any $k \le N$, but is in $C_\delta$, and hence $x$ and $h(x)$ are in $\epsilon/3$-ball around $c$ -- contradicting $d(h(x),x) \ge \epsilon$. 

If $i=j\le N$ then $\cl{I}$ is a circle containing $c_i$. Then $h(\cl{I})$ is also a circle containing $I$, and so meets $B_\delta(c_i)$. By conditions (iii) and (iv) $h(\cl{I})$ meets no  part of $C_\delta$. Hence $h(x)$ must be in $B_\delta(c_i)$, and $d(h(x),x) < \epsilon$, contradiction.

Otherwise, at least one of $i$ and $j$ is $\le N$. Say $i$. By condition (v), the arc $A$, which is the subinterval of the arc $\cl{I}$ obtained by starting at $c_i$ and traveling along the arc to $x$, is contained in $B_\delta(c_i)$. Then $h(A)$ is an arc traveling from $c_i$ to $h(x)$ which is in $C_\delta$ but not in $B_\delta(c_i)$. So let $y'$ be the point on the boundary of $C_\delta$ (not in $C_\delta$) when $h(A)$ leaves the complement of $C_\delta$ near $h(x)$. Then, as $h^{-1}(y')$ is in $A$ which is contained in $B_\delta(c_i)$ but $y'$ is near $h(x)$, we see  $d(y',h^{-1}(y')) > \delta$, contradicting the hypothesis on $h$.
\end{proof}

The following examples are  all easily seen to satisfy  condition (a) above and indeed have finite $C$.

\begin{eg} \ 

(a) The homeomorphism group of a convergent sequence of circles (with unique limit)  has the minimum (and hence minimal) group topology, which is the compact-open topology.

(b) For every finite graph, the homeomorphism group with the compact-open topology is the minimum group topology (and so minimal).

(c) The homeomorphism group of the Hawaiian earring  has the minimum (and hence minimal) group topology, which is the compact-open topology.

(d) More generally, if $C$ is finite and $S_X$ has a unique limit point in $C$, then $X \setminus S_X$ has only finitely many components and $X$ can be obtained in the following way. Take any finite graph, which need not be connected, and may have multiple edges between vertices and from a vertex back to itself. Possibly `decorate' one vertex by adding a sequence of circles converging to that vertex. Possibly decorate any other vertex by adding a Hawaiian earring at that vertex. Then the homeomorphism group of such a space has the minimum group topology.
\end{eg}

Taking a convergent sequence, along with the limit point, of finite graphs or Hawaiian earrings (or spaces as in (d) above, but with no sequence of circles) will give spaces satisfying condition (b) of the preceding proposition.

Many examples also arise when $C$ is uncountable.
\begin{eg}\label{CantorB} \ 

(a) The homeomorphism group of the Cantor bouquet of semi-circles has the minimum (and hence minimal) group topology, which is the compact-open topology.

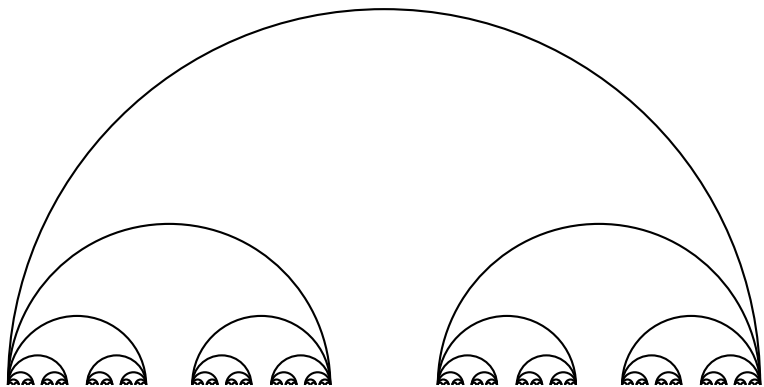
\begin{center}
\begin{tikzpicture}

\semicircle{5.0}{0}{5.0}
\semicircle{2.14285714286}{0}{2.14285714286}
\semicircle{7.85714285714}{0}{2.14285714286}
\semicircle{0.918367346939}{0}{0.918367346939}
\semicircle{6.63265306122}{0}{0.918367346939}
\semicircle{3.36734693878}{0}{0.918367346939}
\semicircle{9.08163265306}{0}{0.918367346939}
\semicircle{0.393586005831}{0}{0.393586005831}
\semicircle{6.10787172012}{0}{0.393586005831}
\semicircle{2.84256559767}{0}{0.393586005831}
\semicircle{8.55685131195}{0}{0.393586005831}
\semicircle{1.44314868805}{0}{0.393586005831}
\semicircle{7.15743440233}{0}{0.393586005831}
\semicircle{3.89212827988}{0}{0.393586005831}
\semicircle{9.60641399417}{0}{0.393586005831}
\semicircle{0.168679716785}{0}{0.168679716785}
\semicircle{5.88296543107}{0}{0.168679716785}
\semicircle{2.61765930862}{0}{0.168679716785}
\semicircle{8.33194502291}{0}{0.168679716785}
\semicircle{1.218242399}{0}{0.168679716785}
\semicircle{6.93252811329}{0}{0.168679716785}
\semicircle{3.66722199084}{0}{0.168679716785}
\semicircle{9.38150770512}{0}{0.168679716785}
\semicircle{0.618492294877}{0}{0.168679716785}
\semicircle{6.33277800916}{0}{0.168679716785}
\semicircle{3.06747188671}{0}{0.168679716785}
\semicircle{8.781757601}{0}{0.168679716785}
\semicircle{1.66805497709}{0}{0.168679716785}
\semicircle{7.38234069138}{0}{0.168679716785}
\semicircle{4.11703456893}{0}{0.168679716785}
\semicircle{9.83132028322}{0}{0.168679716785}
\semicircle{0.0722913071934}{0}{0.0722913071934}
\semicircle{5.78657702148}{0}{0.0722913071934}
\semicircle{2.52127089903}{0}{0.0722913071934}
\semicircle{8.23555661332}{0}{0.0722913071934}
\semicircle{1.12185398941}{0}{0.0722913071934}
\semicircle{6.83613970369}{0}{0.0722913071934}
\semicircle{3.57083358125}{0}{0.0722913071934}
\semicircle{9.28511929553}{0}{0.0722913071934}
\semicircle{0.522103885286}{0}{0.0722913071934}
\semicircle{6.23638959957}{0}{0.0722913071934}
\semicircle{2.97108347712}{0}{0.0722913071934}
\semicircle{8.68536919141}{0}{0.0722913071934}
\semicircle{1.5716665675}{0}{0.0722913071934}
\semicircle{7.28595228179}{0}{0.0722913071934}
\semicircle{4.02064615934}{0}{0.0722913071934}
\semicircle{9.73493187362}{0}{0.0722913071934}
\semicircle{0.265068126376}{0}{0.0722913071934}
\semicircle{5.97935384066}{0}{0.0722913071934}
\semicircle{2.71404771821}{0}{0.0722913071934}
\semicircle{8.4283334325}{0}{0.0722913071934}
\semicircle{1.31463080859}{0}{0.0722913071934}
\semicircle{7.02891652288}{0}{0.0722913071934}
\semicircle{3.76361040043}{0}{0.0722913071934}
\semicircle{9.47789611471}{0}{0.0722913071934}
\semicircle{0.714880704468}{0}{0.0722913071934}
\semicircle{6.42916641875}{0}{0.0722913071934}
\semicircle{3.16386029631}{0}{0.0722913071934}
\semicircle{8.87814601059}{0}{0.0722913071934}
\semicircle{1.76444338668}{0}{0.0722913071934}
\semicircle{7.47872910097}{0}{0.0722913071934}
\semicircle{4.21342297852}{0}{0.0722913071934}
\semicircle{9.92770869281}{0}{0.0722913071934}
\semicircle{0.0309819887972}{0}{0.0309819887972}
\semicircle{5.74526770308}{0}{0.0309819887972}
\semicircle{2.47996158063}{0}{0.0309819887972}
\semicircle{8.19424729492}{0}{0.0309819887972}
\semicircle{1.08054467101}{0}{0.0309819887972}
\semicircle{6.7948303853}{0}{0.0309819887972}
\semicircle{3.52952426285}{0}{0.0309819887972}
\semicircle{9.24380997714}{0}{0.0309819887972}
\semicircle{0.48079456689}{0}{0.0309819887972}
\semicircle{6.19508028118}{0}{0.0309819887972}
\semicircle{2.92977415873}{0}{0.0309819887972}
\semicircle{8.64405987301}{0}{0.0309819887972}
\semicircle{1.53035724911}{0}{0.0309819887972}
\semicircle{7.24464296339}{0}{0.0309819887972}
\semicircle{3.97933684094}{0}{0.0309819887972}
\semicircle{9.69362255523}{0}{0.0309819887972}
\semicircle{0.22375880798}{0}{0.0309819887972}
\semicircle{5.93804452227}{0}{0.0309819887972}
\semicircle{2.67273839982}{0}{0.0309819887972}
\semicircle{8.3870241141}{0}{0.0309819887972}
\semicircle{1.2733214902}{0}{0.0309819887972}
\semicircle{6.98760720448}{0}{0.0309819887972}
\semicircle{3.72230108203}{0}{0.0309819887972}
\semicircle{9.43658679632}{0}{0.0309819887972}
\semicircle{0.673571386072}{0}{0.0309819887972}
\semicircle{6.38785710036}{0}{0.0309819887972}
\semicircle{3.12255097791}{0}{0.0309819887972}
\semicircle{8.83683669219}{0}{0.0309819887972}
\semicircle{1.72313406829}{0}{0.0309819887972}
\semicircle{7.43741978257}{0}{0.0309819887972}
\semicircle{4.17211366012}{0}{0.0309819887972}
\semicircle{9.88639937441}{0}{0.0309819887972}
\semicircle{0.11360062559}{0}{0.0309819887972}
\semicircle{5.82788633988}{0}{0.0309819887972}
\semicircle{2.56258021743}{0}{0.0309819887972}
\semicircle{8.27686593171}{0}{0.0309819887972}
\semicircle{1.16316330781}{0}{0.0309819887972}
\semicircle{6.87744902209}{0}{0.0309819887972}
\semicircle{3.61214289964}{0}{0.0309819887972}
\semicircle{9.32642861393}{0}{0.0309819887972}
\semicircle{0.563413203682}{0}{0.0309819887972}
\semicircle{6.27769891797}{0}{0.0309819887972}
\semicircle{3.01239279552}{0}{0.0309819887972}
\semicircle{8.7266785098}{0}{0.0309819887972}
\semicircle{1.6129758859}{0}{0.0309819887972}
\semicircle{7.32726160018}{0}{0.0309819887972}
\semicircle{4.06195547773}{0}{0.0309819887972}
\semicircle{9.77624119202}{0}{0.0309819887972}
\semicircle{0.306377444772}{0}{0.0309819887972}
\semicircle{6.02066315906}{0}{0.0309819887972}
\semicircle{2.75535703661}{0}{0.0309819887972}
\semicircle{8.46964275089}{0}{0.0309819887972}
\semicircle{1.35594012699}{0}{0.0309819887972}
\semicircle{7.07022584127}{0}{0.0309819887972}
\semicircle{3.80491971882}{0}{0.0309819887972}
\semicircle{9.51920543311}{0}{0.0309819887972}
\semicircle{0.756190022865}{0}{0.0309819887972}
\semicircle{6.47047573715}{0}{0.0309819887972}
\semicircle{3.2051696147}{0}{0.0309819887972}
\semicircle{8.91945532899}{0}{0.0309819887972}
\semicircle{1.80575270508}{0}{0.0309819887972}
\semicircle{7.52003841937}{0}{0.0309819887972}
\semicircle{4.25473229692}{0}{0.0309819887972}
\semicircle{9.9690180112}{0}{0.0309819887972}

\end{tikzpicture}
\end{center}

(b) More generally, for any compact, connected  metrizable graph-like space $X$, we see that $H(X)$ with the compact-open topology is the minimum group topology.
\end{eg}

From (a) we deduce a general result. 

\begin{lem}\label{lotsExs}
For every compact metric space $K$ there is a compact, metric connected $X=X_K$ such that $O_X$ is dense,  $C_X=K$, and $\tau_{k|C}=\tau_k$. 
\end{lem}
\begin{proof}
Let $Z$ be the space (illustrated above) from Example~\ref{CantorB}(a). Every compact metric space $K$ is the continuous image of the Cantor set. Applying this quotient to the copy of the Cantor set in $Z$ gives $X=X_K$. It is clear that $O_X$ is dense and $C_X=K$. Modifying the argument for $Z$ easily shows that $\tau_{k|C}=\tau_k$  for $X_K$. 
\end{proof}

\subsection{The Remaining Cases}
Let $X$ be compact metrizable, and suppose $O_X$ is dense in $X$. Let $C=C_X$.
From the preceding two Propositions we (essentially) know whether or not $\tau_{k|C} = \tau_k$ except in two cases: (i) when $S_X$ is empty, $C$ is zero dimensional but not a convergent sequence (or finite) and (ii) $S_X$ is empty, $C$ is not zero dimensional but every component $I$ in $I_X$ has closure either an arc or a circle. Of particular interest in case (ii) is when $C$ and/or $X$ is connected, and preferably locally connected.

In each case there do not seem to be simple conditions we can place on $X$, $C$ etc that allow us to determine if $\tau_{k|C}$ is, or is not, equal to $\tau_k$. We demonstrate this with some examples, starting with case (i).

\begin{eg} \ 

(a) Examples with $C$ two convergent sequences and $\tau_{k|C}\ne\tau_k$:

two convergent sequences of Hawaiian earrings (illustrated below); two convergent sequences of arcs; two convergent sequences of double circles.

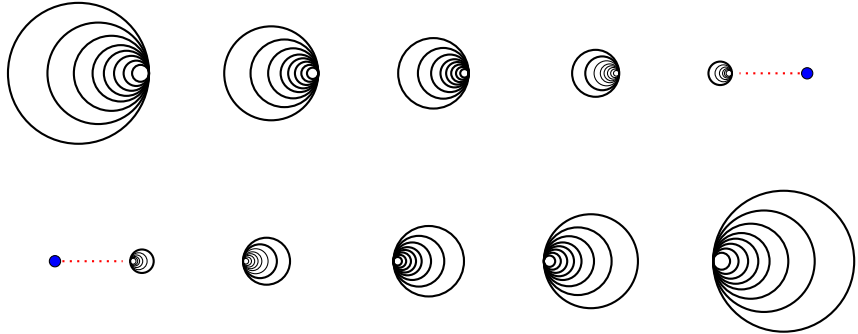
\begin{center}

\begin{tikzpicture}

\draw[dotted,red,thick] (0,0) -- (0.9,0);

\draw[dotted,red,thick] (10,2.5) -- (9.1,2.5);

\draw[fill=blue] (0,0) circle [radius=0.075];

\draw[fill=blue] (10,2.5) circle [radius=0.075];

\draw[thick] (8.75,0) arc [start angle=180, end angle=540, radius=0.9375];
\draw[thick] (8.75,0) arc [start angle=180, end angle=540, radius=0.675];
\draw[thick] (8.75,0) arc [start angle=180, end angle=540, radius=0.4995];
\draw[thick] (8.75,0) arc [start angle=180, end angle=540, radius=0.375];
\draw[thick] (8.75,0) arc [start angle=180, end angle=540, radius=0.3];
\draw[thick] (8.75,0) arc [start angle=180, end angle=540, radius=0.2325];
\draw[thick] (8.75,0) arc [start angle=180, end angle=540, radius=0.16875];
\draw[thick] (8.75,0) arc [start angle=180, end angle=540, radius=0.1125];
\draw[thick] (6.5,0) arc [start angle=180, end angle=540, radius=0.625];
\draw[thick] (6.5,0) arc [start angle=180, end angle=540, radius=0.45];
\draw[thick] (6.5,0) arc [start angle=180, end angle=540, radius=0.333];
\draw[thick] (6.5,0) arc [start angle=180, end angle=540, radius=0.25];
\draw[thick] (6.5,0) arc [start angle=180, end angle=540, radius=0.2];
\draw[thick] (6.5,0) arc [start angle=180, end angle=540, radius=0.155];
\draw[thick] (6.5,0) arc [start angle=180, end angle=540, radius=0.1125];
\draw[thick] (6.5,0) arc [start angle=180, end angle=540, radius=0.075];
\draw[thick] (4.5,0) arc [start angle=180, end angle=540, radius=0.46875];
\draw[thick] (4.5,0) arc [start angle=180, end angle=540, radius=0.3375];
\draw[thick] (4.5,0) arc [start angle=180, end angle=540, radius=0.24975];
\draw[thick] (4.5,0) arc [start angle=180, end angle=540, radius=0.1875];
\draw[thick] (4.5,0) arc [start angle=180, end angle=540, radius=0.15];
\draw[thick] (4.5,0) arc [start angle=180, end angle=540, radius=0.11625];
\draw[thick] (4.5,0) arc [start angle=180, end angle=540, radius=0.084375];
\draw[thick] (4.5,0) arc [start angle=180, end angle=540, radius=0.05625];
\draw[thick] (2.5,0) arc [start angle=180, end angle=540, radius=0.3125];
\draw[thick] (2.5,0) arc [start angle=180, end angle=540, radius=0.225];
\draw (2.5,0) arc [start angle=180, end angle=540, radius=0.1665];
\draw (2.5,0) arc [start angle=180, end angle=540, radius=0.125];
\draw (2.5,0) arc [start angle=180, end angle=540, radius=0.1];
\draw (2.5,0) arc [start angle=180, end angle=540, radius=0.0775];
\draw (2.5,0) arc [start angle=180, end angle=540, radius=0.05625];
\draw (2.5,0) arc [start angle=180, end angle=540, radius=0.0375];
\draw[thick] (1,0) arc [start angle=180, end angle=540, radius=0.15625];
\draw (1,0) arc [start angle=180, end angle=540, radius=0.1125];
\draw (1,0) arc [start angle=180, end angle=540, radius=0.08325];
\draw (1,0) arc [start angle=180, end angle=540, radius=0.0625];
\draw (1,0) arc [start angle=180, end angle=540, radius=0.05];
\draw (1,0) arc [start angle=180, end angle=540, radius=0.03875];

\draw[thick] (1.25,2.5) arc [start angle=0, end angle=360, radius=0.9375];
\draw[thick] (1.25,2.5) arc [start angle=0, end angle=360, radius=0.675];
\draw[thick] (1.25,2.5) arc [start angle=0, end angle=360, radius=0.4995];
\draw[thick] (1.25,2.5) arc [start angle=0, end angle=360, radius=0.375];
\draw[thick] (1.25,2.5) arc [start angle=0, end angle=360, radius=0.3];
\draw[thick] (1.25,2.5) arc [start angle=0, end angle=360, radius=0.2325];
\draw[thick] (1.25,2.5) arc [start angle=0, end angle=360, radius=0.16875];
\draw[thick] (1.25,2.5) arc [start angle=0, end angle=360, radius=0.1125];
\draw[thick] (3.5,2.5) arc [start angle=0, end angle=360, radius=0.625];
\draw[thick] (3.5,2.5) arc [start angle=0, end angle=360, radius=0.45];
\draw[thick] (3.5,2.5) arc [start angle=0, end angle=360, radius=0.333];
\draw[thick] (3.5,2.5) arc [start angle=0, end angle=360, radius=0.25];
\draw[thick] (3.5,2.5) arc [start angle=0, end angle=360, radius=0.2];
\draw[thick] (3.5,2.5) arc [start angle=0, end angle=360, radius=0.155];
\draw[thick] (3.5,2.5) arc [start angle=0, end angle=360, radius=0.1125];
\draw[thick] (3.5,2.5) arc [start angle=0, end angle=360, radius=0.075];
\draw[thick] (5.5,2.5) arc [start angle=0, end angle=360, radius=0.46875];
\draw[thick] (5.5,2.5) arc [start angle=0, end angle=360, radius=0.3375];
\draw[thick] (5.5,2.5) arc [start angle=0, end angle=360, radius=0.24975];
\draw[thick] (5.5,2.5) arc [start angle=0, end angle=360, radius=0.1875];
\draw[thick] (5.5,2.5) arc [start angle=0, end angle=360, radius=0.15];
\draw[thick] (5.5,2.5) arc [start angle=0, end angle=360, radius=0.11625];
\draw[thick] (5.5,2.5) arc [start angle=0, end angle=360, radius=0.084375];
\draw[thick] (5.5,2.5) arc [start angle=0, end angle=360, radius=0.05625];
\draw[thick] (7.5,2.5) arc [start angle=0, end angle=360, radius=0.3125];
\draw[thick] (7.5,2.5) arc [start angle=0, end angle=360, radius=0.225];
\draw (7.5,2.5) arc [start angle=0, end angle=360, radius=0.1665];
\draw (7.5,2.5) arc [start angle=0, end angle=360, radius=0.125];
\draw (7.5,2.5) arc [start angle=0, end angle=360, radius=0.1];
\draw (7.5,2.5) arc [start angle=0, end angle=360, radius=0.0775];
\draw (7.5,2.5) arc [start angle=0, end angle=360, radius=0.05625];
\draw (7.5,2.5) arc [start angle=0, end angle=360, radius=0.0375];
\draw[thick] (9,2.5) arc [start angle=0, end angle=360, radius=0.15625];
\draw (9,2.5) arc [start angle=0, end angle=360, radius=0.1125];
\draw (9,2.5) arc [start angle=0, end angle=360, radius=0.08325];
\draw (9,2.5) arc [start angle=0, end angle=360, radius=0.0625];
\draw (9,2.5) arc [start angle=0, end angle=360, radius=0.05];
\draw (9,2.5) arc [start angle=0, end angle=360, radius=0.03875];

\end{tikzpicture}

\end{center}

(a)${}'$ Examples with $C$ two convergent sequences but $\tau_{k|C} = \tau_k$:

a convergent sequence of $2n$-circles plus a convergent sequence of $2n+1$-circles; a convergent sequence of $2n$-ods plus a convergent sequence of $2n+1$-ods (illustrated below).

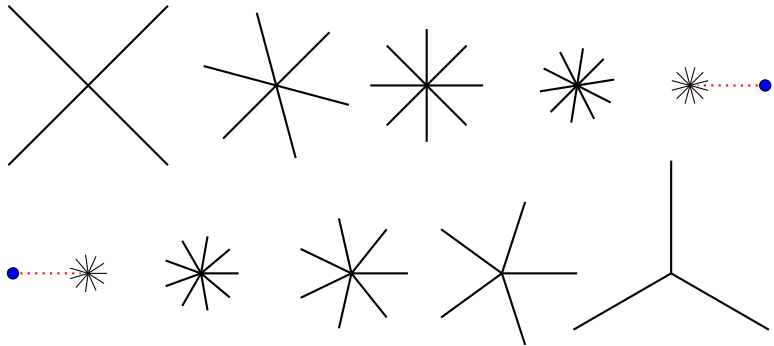
\begin{center}

\begin{tikzpicture}

\draw[dotted,red,thick] (0,0) -- (0.9,0);

\draw[fill=blue] (0,0)  circle (0.075);

\draw[dotted,red,thick] (10,2.5) -- (9.1,2.5);

\draw[fill=blue] (10,2.5)  circle (0.075);

\draw[thick] (1,2.5) --++(135.0:1.5);
\draw[thick] (1,2.5) --++(225.0:1.5);
\draw[thick] (1,2.5) --++(315.0:1.5);
\draw[thick] (1,2.5) --++(405.0:1.5);
\draw[thick] (3.5,2.5) --++(105.0:1);
\draw[thick] (3.5,2.5) --++(165.0:1);
\draw[thick] (3.5,2.5) --++(225.0:1);
\draw[thick] (3.5,2.5) --++(285.0:1);
\draw[thick] (3.5,2.5) --++(345.0:1);
\draw[thick] (3.5,2.5) --++(405.0:1);
\draw[thick] (5.5,2.5) --++(90.0:0.75);
\draw[thick] (5.5,2.5) --++(135.0:0.75);
\draw[thick] (5.5,2.5) --++(180.0:0.75);
\draw[thick] (5.5,2.5) --++(225.0:0.75);
\draw[thick] (5.5,2.5) --++(270.0:0.75);
\draw[thick] (5.5,2.5) --++(315.0:0.75);
\draw[thick] (5.5,2.5) --++(360.0:0.75);
\draw[thick] (5.5,2.5) --++(405.0:0.75);
\draw[thick] (7.5,2.5) --++(81.0:0.5);
\draw[thick] (7.5,2.5) --++(117.0:0.5);
\draw[thick] (7.5,2.5) --++(153.0:0.5);
\draw[thick] (7.5,2.5) --++(189.0:0.5);
\draw[thick] (7.5,2.5) --++(225.0:0.5);
\draw[thick] (7.5,2.5) --++(261.0:0.5);
\draw[thick] (7.5,2.5) --++(297.0:0.5);
\draw[thick] (7.5,2.5) --++(333.0:0.5);
\draw[thick] (7.5,2.5) --++(369.0:0.5);
\draw[thick] (7.5,2.5) --++(405.0:0.5);
\draw (9,2.5) --++(75.0:0.25);
\draw (9,2.5) --++(105.0:0.25);
\draw (9,2.5) --++(135.0:0.25);
\draw (9,2.5) --++(165.0:0.25);
\draw (9,2.5) --++(195.0:0.25);
\draw (9,2.5) --++(225.0:0.25);
\draw (9,2.5) --++(255.0:0.25);
\draw (9,2.5) --++(285.0:0.25);
\draw (9,2.5) --++(315.0:0.25);
\draw (9,2.5) --++(345.0:0.25);
\draw (9,2.5) --++(375.0:0.25);
\draw (9,2.5) --++(405.0:0.25);

\draw[thick] (8.75,0) --++(210.0:1.5);
\draw[thick] (8.75,0) --++(330.0:1.5);
\draw[thick] (8.75,0) --++(90.0:1.5);
\draw[thick] (6.5,0) --++(72.0:1);
\draw[thick] (6.5,0) --++(144.0:1);
\draw[thick] (6.5,0) --++(216.0:1);
\draw[thick] (6.5,0) --++(288.0:1);
\draw[thick] (6.5,0) --++(360.0:1);
\draw[thick] (4.5,0) --++(51.4285714286:0.75);
\draw[thick] (4.5,0) --++(102.857142857:0.75);
\draw[thick] (4.5,0) --++(154.285714286:0.75);
\draw[thick] (4.5,0) --++(205.714285714:0.75);
\draw[thick] (4.5,0) --++(257.142857143:0.75);
\draw[thick] (4.5,0) --++(308.571428571:0.75);
\draw[thick] (4.5,0) --++(360.0:0.75);
\draw[thick] (2.5,0) --++(40.0:0.5);
\draw[thick] (2.5,0) --++(80.0:0.5);
\draw[thick] (2.5,0) --++(120.0:0.5);
\draw[thick] (2.5,0) --++(160.0:0.5);
\draw[thick] (2.5,0) --++(200.0:0.5);
\draw[thick] (2.5,0) --++(240.0:0.5);
\draw[thick] (2.5,0) --++(280.0:0.5);
\draw[thick] (2.5,0) --++(320.0:0.5);
\draw[thick] (2.5,0) --++(360.0:0.5);
\draw (1,0) --++(32.7272727273:0.25);
\draw (1,0) --++(65.4545454545:0.25);
\draw (1,0) --++(98.1818181818:0.25);
\draw (1,0) --++(130.909090909:0.25);
\draw (1,0) --++(163.636363636:0.25);
\draw (1,0) --++(196.363636364:0.25);
\draw (1,0) --++(229.090909091:0.25);
\draw (1,0) --++(261.818181818:0.25);
\draw (1,0) --++(294.545454545:0.25);
\draw (1,0) --++(327.272727273:0.25);
\draw (1,0) --++(360.0:0.25);

\end{tikzpicture}
\end{center}
\end{eg}

Similar examples exist when $C$ is the Cantor set.

So the remaining interesting case is when $X$ and $C$ are connected, $S_X$ is empty and every $I$ in $I_X$ has closure an arc or a circle.

\begin{eg} \ Examples where $\tau_{k|C} = \tau_k$.

(a) Bouquet of circles over closed interval. In this case both $X$ and $C$ are connected and locally connected.

(b) More generally, apply Lemma~\ref{lotsExs} to any continuum $K$ to get many examples with both $X$ and $C$ connected.
\end{eg}

We conclude with two examples when $\tau_{k|C} \ne \tau_k$. Both are such that $X$ is connected, and the second is additionally locally connected. We include both as the first example is a stepping stone to the second. 
\begin{eg} Path connected but not locally connected example where $\tau_{k|C} \ne \tau_k$.

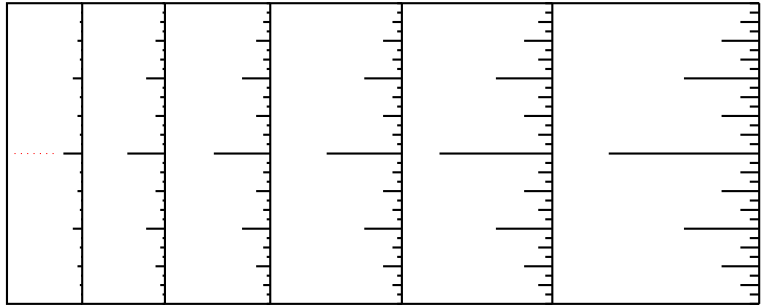
\begin{center}
\begin{tikzpicture} 

\draw[thick] (10.0,0.0) -- (10.0,4.0);
\draw [thick] (10.0,0.0) -- (9.875,0.0);
\draw [thick] (10.0,0.125) -- (9.875,0.125);
\draw [thick] (10.0,0.25) -- (9.75,0.25);
\draw [thick] (10.0,0.375) -- (9.875,0.375);
\draw [thick] (10.0,0.5) -- (9.5,0.5);
\draw [thick] (10.0,0.625) -- (9.875,0.625);
\draw [thick] (10.0,0.75) -- (9.75,0.75);
\draw [thick] (10.0,0.875) -- (9.875,0.875);
\draw [thick] (10.0,1.0) -- (9.0,1.0);
\draw [thick] (10.0,1.125) -- (9.875,1.125);
\draw [thick] (10.0,1.25) -- (9.75,1.25);
\draw [thick] (10.0,1.375) -- (9.875,1.375);
\draw [thick] (10.0,1.5) -- (9.5,1.5);
\draw [thick] (10.0,1.625) -- (9.875,1.625);
\draw [thick] (10.0,1.75) -- (9.75,1.75);
\draw [thick] (10.0,1.875) -- (9.875,1.875);
\draw [thick] (10.0,2.0) -- (8.0,2.0);
\draw [thick] (10.0,2.125) -- (9.875,2.125);
\draw [thick] (10.0,2.25) -- (9.75,2.25);
\draw [thick] (10.0,2.375) -- (9.875,2.375);
\draw [thick] (10.0,2.5) -- (9.5,2.5);
\draw [thick] (10.0,2.625) -- (9.875,2.625);
\draw [thick] (10.0,2.75) -- (9.75,2.75);
\draw [thick] (10.0,2.875) -- (9.875,2.875);
\draw [thick] (10.0,3.0) -- (9.0,3.0);
\draw [thick] (10.0,3.125) -- (9.875,3.125);
\draw [thick] (10.0,3.25) -- (9.75,3.25);
\draw [thick] (10.0,3.375) -- (9.875,3.375);
\draw [thick] (10.0,3.5) -- (9.5,3.5);
\draw [thick] (10.0,3.625) -- (9.875,3.625);
\draw [thick] (10.0,3.75) -- (9.75,3.75);
\draw [thick] (10.0,3.875) -- (9.875,3.875);
\draw [thick] (10.0,4.0) -- (9.875,4.0);
\draw[thick] (7.25,0.0) -- (7.25,4.0);
\draw [thick] (7.25,0.0) -- (7.15625,0.0);
\draw [thick] (7.25,0.125) -- (7.15625,0.125);
\draw [thick] (7.25,0.25) -- (7.0625,0.25);
\draw [thick] (7.25,0.375) -- (7.15625,0.375);
\draw [thick] (7.25,0.5) -- (6.875,0.5);
\draw [thick] (7.25,0.625) -- (7.15625,0.625);
\draw [thick] (7.25,0.75) -- (7.0625,0.75);
\draw [thick] (7.25,0.875) -- (7.15625,0.875);
\draw [thick] (7.25,1.0) -- (6.5,1.0);
\draw [thick] (7.25,1.125) -- (7.15625,1.125);
\draw [thick] (7.25,1.25) -- (7.0625,1.25);
\draw [thick] (7.25,1.375) -- (7.15625,1.375);
\draw [thick] (7.25,1.5) -- (6.875,1.5);
\draw [thick] (7.25,1.625) -- (7.15625,1.625);
\draw [thick] (7.25,1.75) -- (7.0625,1.75);
\draw [thick] (7.25,1.875) -- (7.15625,1.875);
\draw [thick] (7.25,2.0) -- (5.75,2.0);
\draw [thick] (7.25,2.125) -- (7.15625,2.125);
\draw [thick] (7.25,2.25) -- (7.0625,2.25);
\draw [thick] (7.25,2.375) -- (7.15625,2.375);
\draw [thick] (7.25,2.5) -- (6.875,2.5);
\draw [thick] (7.25,2.625) -- (7.15625,2.625);
\draw [thick] (7.25,2.75) -- (7.0625,2.75);
\draw [thick] (7.25,2.875) -- (7.15625,2.875);
\draw [thick] (7.25,3.0) -- (6.5,3.0);
\draw [thick] (7.25,3.125) -- (7.15625,3.125);
\draw [thick] (7.25,3.25) -- (7.0625,3.25);
\draw [thick] (7.25,3.375) -- (7.15625,3.375);
\draw [thick] (7.25,3.5) -- (6.875,3.5);
\draw [thick] (7.25,3.625) -- (7.15625,3.625);
\draw [thick] (7.25,3.75) -- (7.0625,3.75);
\draw [thick] (7.25,3.875) -- (7.15625,3.875);
\draw [thick] (7.25,4.0) -- (7.15625,4.0);
\draw[thick] (5.25,0.0) -- (5.25,4.0);
\draw [thick] (5.25,0.0) -- (5.1875,0.0);
\draw [thick] (5.25,0.125) -- (5.1875,0.125);
\draw [thick] (5.25,0.25) -- (5.125,0.25);
\draw [thick] (5.25,0.375) -- (5.1875,0.375);
\draw [thick] (5.25,0.5) -- (5.0,0.5);
\draw [thick] (5.25,0.625) -- (5.1875,0.625);
\draw [thick] (5.25,0.75) -- (5.125,0.75);
\draw [thick] (5.25,0.875) -- (5.1875,0.875);
\draw [thick] (5.25,1.0) -- (4.75,1.0);
\draw [thick] (5.25,1.125) -- (5.1875,1.125);
\draw [thick] (5.25,1.25) -- (5.125,1.25);
\draw [thick] (5.25,1.375) -- (5.1875,1.375);
\draw [thick] (5.25,1.5) -- (5.0,1.5);
\draw [thick] (5.25,1.625) -- (5.1875,1.625);
\draw [thick] (5.25,1.75) -- (5.125,1.75);
\draw [thick] (5.25,1.875) -- (5.1875,1.875);
\draw [thick] (5.25,2.0) -- (4.25,2.0);
\draw [thick] (5.25,2.125) -- (5.1875,2.125);
\draw [thick] (5.25,2.25) -- (5.125,2.25);
\draw [thick] (5.25,2.375) -- (5.1875,2.375);
\draw [thick] (5.25,2.5) -- (5.0,2.5);
\draw [thick] (5.25,2.625) -- (5.1875,2.625);
\draw [thick] (5.25,2.75) -- (5.125,2.75);
\draw [thick] (5.25,2.875) -- (5.1875,2.875);
\draw [thick] (5.25,3.0) -- (4.75,3.0);
\draw [thick] (5.25,3.125) -- (5.1875,3.125);
\draw [thick] (5.25,3.25) -- (5.125,3.25);
\draw [thick] (5.25,3.375) -- (5.1875,3.375);
\draw [thick] (5.25,3.5) -- (5.0,3.5);
\draw [thick] (5.25,3.625) -- (5.1875,3.625);
\draw [thick] (5.25,3.75) -- (5.125,3.75);
\draw [thick] (5.25,3.875) -- (5.1875,3.875);
\draw [thick] (5.25,4.0) -- (5.1875,4.0);
\draw[thick] (3.5,0.0) -- (3.5,4.0);
\draw [thick] (3.5,0.0) -- (3.453125,0.0);
\draw [thick] (3.5,0.125) -- (3.453125,0.125);
\draw [thick] (3.5,0.25) -- (3.40625,0.25);
\draw [thick] (3.5,0.375) -- (3.453125,0.375);
\draw [thick] (3.5,0.5) -- (3.3125,0.5);
\draw [thick] (3.5,0.625) -- (3.453125,0.625);
\draw [thick] (3.5,0.75) -- (3.40625,0.75);
\draw [thick] (3.5,0.875) -- (3.453125,0.875);
\draw [thick] (3.5,1.0) -- (3.125,1.0);
\draw [thick] (3.5,1.125) -- (3.453125,1.125);
\draw [thick] (3.5,1.25) -- (3.40625,1.25);
\draw [thick] (3.5,1.375) -- (3.453125,1.375);
\draw [thick] (3.5,1.5) -- (3.3125,1.5);
\draw [thick] (3.5,1.625) -- (3.453125,1.625);
\draw [thick] (3.5,1.75) -- (3.40625,1.75);
\draw [thick] (3.5,1.875) -- (3.453125,1.875);
\draw [thick] (3.5,2.0) -- (2.75,2.0);
\draw [thick] (3.5,2.125) -- (3.453125,2.125);
\draw [thick] (3.5,2.25) -- (3.40625,2.25);
\draw [thick] (3.5,2.375) -- (3.453125,2.375);
\draw [thick] (3.5,2.5) -- (3.3125,2.5);
\draw [thick] (3.5,2.625) -- (3.453125,2.625);
\draw [thick] (3.5,2.75) -- (3.40625,2.75);
\draw [thick] (3.5,2.875) -- (3.453125,2.875);
\draw [thick] (3.5,3.0) -- (3.125,3.0);
\draw [thick] (3.5,3.125) -- (3.453125,3.125);
\draw [thick] (3.5,3.25) -- (3.40625,3.25);
\draw [thick] (3.5,3.375) -- (3.453125,3.375);
\draw [thick] (3.5,3.5) -- (3.3125,3.5);
\draw [thick] (3.5,3.625) -- (3.453125,3.625);
\draw [thick] (3.5,3.75) -- (3.40625,3.75);
\draw [thick] (3.5,3.875) -- (3.453125,3.875);
\draw [thick] (3.5,4.0) -- (3.453125,4.0);
\draw[thick] (2.1,0.0) -- (2.1,4.0);
\draw [thick] (2.1,0.0) -- (2.06875,0.0);
\draw [thick] (2.1,0.125) -- (2.06875,0.125);
\draw [thick] (2.1,0.25) -- (2.0375,0.25);
\draw [thick] (2.1,0.375) -- (2.06875,0.375);
\draw [thick] (2.1,0.5) -- (1.975,0.5);
\draw [thick] (2.1,0.625) -- (2.06875,0.625);
\draw [thick] (2.1,0.75) -- (2.0375,0.75);
\draw [thick] (2.1,0.875) -- (2.06875,0.875);
\draw [thick] (2.1,1.0) -- (1.85,1.0);
\draw [thick] (2.1,1.125) -- (2.06875,1.125);
\draw [thick] (2.1,1.25) -- (2.0375,1.25);
\draw [thick] (2.1,1.375) -- (2.06875,1.375);
\draw [thick] (2.1,1.5) -- (1.975,1.5);
\draw [thick] (2.1,1.625) -- (2.06875,1.625);
\draw [thick] (2.1,1.75) -- (2.0375,1.75);
\draw [thick] (2.1,1.875) -- (2.06875,1.875);
\draw [thick] (2.1,2.0) -- (1.6,2.0);
\draw [thick] (2.1,2.125) -- (2.06875,2.125);
\draw [thick] (2.1,2.25) -- (2.0375,2.25);
\draw [thick] (2.1,2.375) -- (2.06875,2.375);
\draw [thick] (2.1,2.5) -- (1.975,2.5);
\draw [thick] (2.1,2.625) -- (2.06875,2.625);
\draw [thick] (2.1,2.75) -- (2.0375,2.75);
\draw [thick] (2.1,2.875) -- (2.06875,2.875);
\draw [thick] (2.1,3.0) -- (1.85,3.0);
\draw [thick] (2.1,3.125) -- (2.06875,3.125);
\draw [thick] (2.1,3.25) -- (2.0375,3.25);
\draw [thick] (2.1,3.375) -- (2.06875,3.375);
\draw [thick] (2.1,3.5) -- (1.975,3.5);
\draw [thick] (2.1,3.625) -- (2.06875,3.625);
\draw [thick] (2.1,3.75) -- (2.0375,3.75);
\draw [thick] (2.1,3.875) -- (2.06875,3.875);
\draw [thick] (2.1,4.0) -- (2.06875,4.0);
\draw[thick] (1,0.0) -- (1,4.0);
\draw [thick] (1,0.0) -- (0.984375,0.0);
\draw [thick] (1,0.125) -- (0.984375,0.125);
\draw [thick] (1,0.25) -- (0.96875,0.25);
\draw [thick] (1,0.375) -- (0.984375,0.375);
\draw [thick] (1,0.5) -- (0.9375,0.5);
\draw [thick] (1,0.625) -- (0.984375,0.625);
\draw [thick] (1,0.75) -- (0.96875,0.75);
\draw [thick] (1,0.875) -- (0.984375,0.875);
\draw [thick] (1,1.0) -- (0.875,1.0);
\draw [thick] (1,1.125) -- (0.984375,1.125);
\draw [thick] (1,1.25) -- (0.96875,1.25);
\draw [thick] (1,1.375) -- (0.984375,1.375);
\draw [thick] (1,1.5) -- (0.9375,1.5);
\draw [thick] (1,1.625) -- (0.984375,1.625);
\draw [thick] (1,1.75) -- (0.96875,1.75);
\draw [thick] (1,1.875) -- (0.984375,1.875);
\draw [thick] (1,2.0) -- (0.75,2.0);
\draw [thick] (1,2.125) -- (0.984375,2.125);
\draw [thick] (1,2.25) -- (0.96875,2.25);
\draw [thick] (1,2.375) -- (0.984375,2.375);
\draw [thick] (1,2.5) -- (0.9375,2.5);
\draw [thick] (1,2.625) -- (0.984375,2.625);
\draw [thick] (1,2.75) -- (0.96875,2.75);
\draw [thick] (1,2.875) -- (0.984375,2.875);
\draw [thick] (1,3.0) -- (0.875,3.0);
\draw [thick] (1,3.125) -- (0.984375,3.125);
\draw [thick] (1,3.25) -- (0.96875,3.25);
\draw [thick] (1,3.375) -- (0.984375,3.375);
\draw [thick] (1,3.5) -- (0.9375,3.5);
\draw [thick] (1,3.625) -- (0.984375,3.625);
\draw [thick] (1,3.75) -- (0.96875,3.75);
\draw [thick] (1,3.875) -- (0.984375,3.875);
\draw [thick] (1,4.0) -- (0.984375,4.0);

\draw[dotted,red] (0.1,2) -- (0.65,2);

\draw[thick] (10,0) -- (0,0) -- (0,4) -- (10,4);

\end{tikzpicture}

\end{center}

\end{eg}

\begin{proof} Each vertical line with attached horizontal lines forms a `rational comb'. By shifting and scaling we can find a homeomorphism of the rational comb taking any given horizontal line to any other.

All the vertical lines are in $C$.  The height of the boundary rectangle is $1$ unit. Set $\epsilon = 1/2$. Then for any $\delta>0$, $C_\delta$ will contain one of the vertical lines distinct from the left edge, call it $V$, and all of the  horizontal lines attached to it (in other words, a complete rational comb). Now we can find a homeomorphism of the continuum which is the identity outside the $\delta$-neighborhood of the vertical line $V$, but inside that $\delta$-neighborhood moves the rational comb around so some point is moved a distance at least $1/2$ (in fact we can move things any distance $<1$). This establishes $\tau_{k|C} \ne \tau_k$.
\end{proof}
Note that we can modify the above example so that the horizontal `teeth' of the rational combs in fact point out of the page. And we can replace each tooth with a circle (an arbitrarily small deformation of the original line segment).

\begin{eg}
Connected and locally connected example with $\tau_{k|C}\ne \tau_k$.
\end{eg}

\begin{proof} We start by constructing the key building block, $B$, of the example. It is obtained by adjoining countably many circles (triangles in the illustration) to the unit square in the $x$-$y$ plane ($[0,1]^2 \times \{0\}$). The circles have radii converging to zero, and are tangent to the points in the unit square whose $x$ and $y$ co-ordinates are both dyadic rationals. See the illustration. Note that each line, $\{x\} \times [0,1] \times \{0\}$, where $x$ is a dyadic rational is homeomorphic to the variant of the rational comb described in the paragraph above.

The key property of our building block, $B$, is that we can find a homeomorphism $h_0$ of $B$ moving some point on the line $\{1/2\} \times I \times \{0\}$ a distance at least $1/2$ along that line, but which is the identity on the boundary of the unit square.  Note also that $B$ is a subset of the triangular cylinder (dotted in the diagram), and $C_B$ is the unit square.  
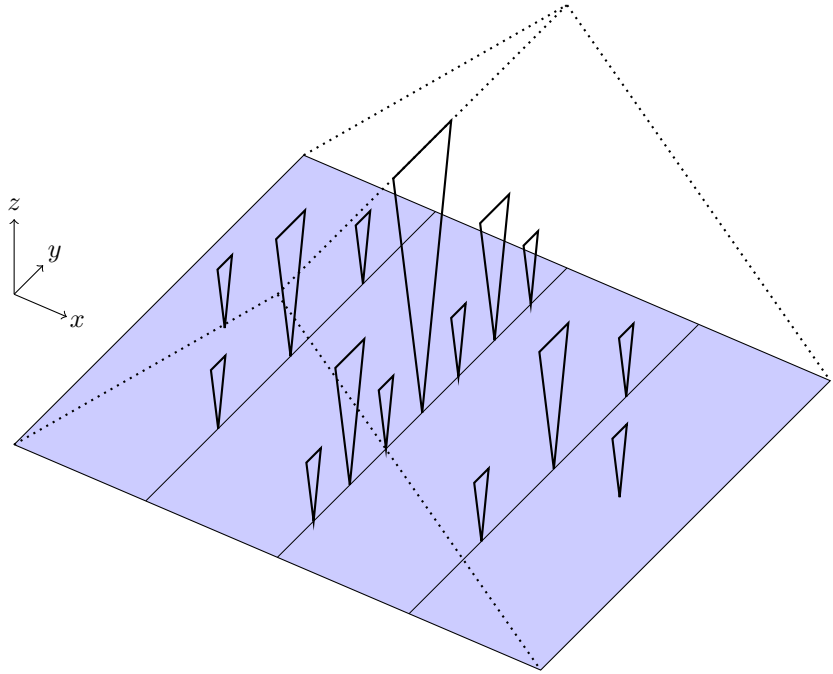
\begin{center}
\begin{tikzpicture}[x={(0.7cm,-0.3cm)},y={(0.385cm,0.385cm)},z={(0cm,1cm)}]

\draw[->] (0,0,2) -- ++ (1,0,0);

\node at (1.2,0,2) {$x$};

\draw[->] (0,0,2) -- ++(0,1,0);

\node at (0,1.4,2) {$y$};

\draw[->] (0,0,2) -- ++(0,0,1);

\node at (0,0,3.2) {$z$};

\draw[fill=blue!20] (0,0) -- (0,10) -- (10,10) -- (10,0)--cycle;

\draw[very thin] (5,0,0) -- (5,10,0);

\draw[very thin] (2.5,0,0) -- (2.5,10,0);

\draw[very thin] (7.5,0,0) -- (7.5,10,0);

\draw[dotted,thick] (0,0,0) -- (5,0,3.5) -- (10,0,0);

\draw[dotted,thick] (0,10,0) -- (5,10,3.5) -- (10,10,0);

\draw[dotted,thick] (5,0,3.5) -- (5,10,3.5);

\draw[thick] (5,5,0) -- (5,4,3.5) -- (5,6,3.5) -- (5,5,0);

\draw[thick] (5,2.5,0) -- (5,2,1.75) -- (5,3,1.75) -- (5,2.5,0);
\draw[thick] (5,7.5,0) -- (5,7,1.75) -- (5,8,1.75) -- (5,7.5,0);
\draw[thick] (2.5,5,0) -- (2.5,4.5,1.75) -- (2.5,5.5,1.75) -- (2.5,5,0);
\draw[thick] (7.5,5,0) -- (7.5,4.5,1.75) -- (7.5,5.5,1.75) -- (7.5,5,0);

\draw[thick] (5,1.25,0) -- (5,1,0.875)--(5,1.5,0.875) -- cycle;;
\draw[thick] (5,3.75,0) -- (5,3.5,0.875)--(5,4,0.875) -- cycle;;
\draw[thick] (5,6.25,0) -- (5,6,0.875)--(5,6.5,0.875) -- cycle;;
\draw[thick] (5,8.75,0) -- (5,8.5,0.875)--(5,9,0.875) -- cycle;;

\draw[thick] (2.5,2.5,0) -- (2.5,2.25,0.875) -- (2.5,2.75,0.875) -- (2.5,2.5,0);
\draw[thick] (2.5,7.5,0) -- (2.5,7.25,0.875) -- (2.5,7.75,0.875) -- (2.5,7.5,0);

\draw[thick] (7.5,2.5,0) -- (7.5,2.25,0.875) -- (7.5,2.75,0.875) -- (7.5,2.5,0);
\draw[thick] (7.5,7.5,0) -- (7.5,7.25,0.875) -- (7.5,7.75,0.875) -- (7.5,7.5,0);

\draw[thick] (8.75,5,0) -- (8.75,4.75,0.875) -- (8.75,5.25,0.875) -- (8.75,5,0);

\draw[thick] (1.25,5,0) -- (1.25,4.75,0.875) -- (1.25,5.25,0.875) -- (1.25,5,0);
\end{tikzpicture}
\end{center}

To obtain the example $X$ start with the unit square in the $x,y$-plane. Take countably many copies, $B_n$, of $B$. Scale (in the $x$ and $z$ directions only) and translate the $B_n$'s to form a sequence, with heights shrinking to zero, converging to the left edge, $I \times \{0,0)\}$ of the unit square.
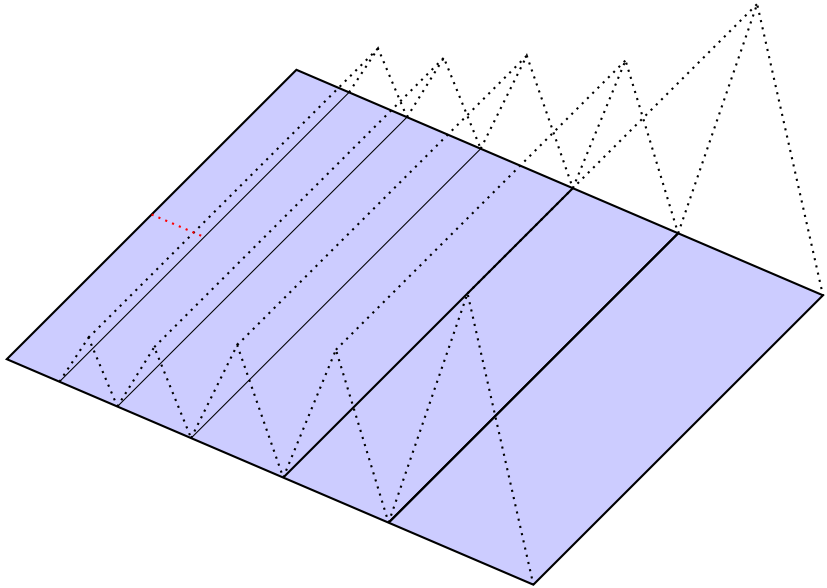
\begin{center}
\begin{tikzpicture}[x={(0.7cm,-0.3cm)},y={(0.385cm,0.385cm)},z={(0cm,1cm)}]

\draw[fill=blue!20,thick] (0,0) -- (0,10) -- (10,10) -- (10,0)--cycle;

\draw[thick] (7.25,0,0) -- (7.25,10,0);

\draw[dotted,thick] (7.25,0,0) -- (8.75,0,3.5) -- (10,0,0);
\draw[dotted,thick] (7.25,10,0) -- (8.75,10,3.5) -- (10,10,0);
\draw[dotted,thick] (8.75,0,3.5) -- (8.75,10,3.5);

\draw[thick] (7.25,0,0) -- (7.25,10,0);

\draw[dotted,thick] (5.25,0,0) -- (6.25,0,2) -- (7.25,0,0);
\draw[dotted,thick] (5.25,10,0) -- (6.25,10,2) -- (7.25,10,0);
\draw[dotted,thick] (6.25,0,2) -- (6.25,10,2);

\draw[thick] (5.25,0,0) -- (5.25,10,0);

\draw[dotted,thick] (3.5,0,0) -- (4.375,0,1.5) -- (5.25,0,0);
\draw[dotted,thick] (3.5,10,0) -- (4.375,10,1.5) -- (5.25,10,0);
\draw[dotted,thick] (4.375,0,1.5) -- (4.375,10,1.5);

\draw  (3.5,0,0) -- (3.5,10,0);

\draw[dotted,thick] (2.1,0,0) -- (2.8,0,1) -- (3.5,0,0);
\draw[dotted,thick] (2.1,10,0) -- (2.8,10,1) -- (3.5,10,0);
\draw[dotted,thick] (2.8,0,1) -- (2.8,10,1);

\draw  (2.1,0,0) -- (2.1,10,0);

\draw[dotted,thick] (1,0,0) -- (1.55,0,0.75) -- (2.1,0,0);
\draw[dotted,thick] (1,10,0) -- (1.55,10,0.75) -- (2.1,10,0);
\draw[dotted,thick] (1.55,0,0.75) -- (1.55,10,0.75);

\draw  (1,0,0) -- (1,10,0);

\draw[dotted,red,thick] (0,5,0) -- (1,5,0);

\end{tikzpicture}
\end{center}
This space  $X$ has the requisite properties. It is connected and locally connected. The set $C_X$ is the unit square in the $x,y$-plane. 
While the proof that $\tau_{k|C}=\tau_k$ is similar to that for the previous example. 

Indeed let $\epsilon=1/4$. Then for any $\delta >0$, we can find a copy of $B$ completely contained in the $\delta$-neighborhood of $C$. Extend the homeomorphism $h_0$ of $B$ over all of $X$ by making it the identity outside $B$. Then this extended homeomorphism, $h_0$, moves nothing outside $C_\delta$, but moves -- in the $y$-direction only -- at least one point a distance $\ge \epsilon$.
\end{proof}

\section{Automorphism Groups}

Let $M$ be a countable model of a first order theory. Then its automorphism group, $\mathop{Aut} (M)$, considered as a topological subgroup of $S(M)$ (the group of all permutations of $M$, with the topology of pointwise convergence) is a closed subgroup, and hence Polish. Conversely, every closed subgroup of $S(\N)$ can be identified as the automorphism group of a countable model of a first order theory. If a theory is $\aleph_0$-categorical -- has a unique (up to isomorphism) countable model -- $M$, say, then $\mathop{Aut}(M)$ is said to be oligomorphic. Equivalently, a closed subgroup of $S(\N)$ is oligomorphic if and only if for each $n \ge 1$ its natural action on $\N^n$ has finitely many orbits. We recall that a subgroup $G$ of $S(\N)$ is highly homogeneous if for any two finite subsets $A,B$ of $\N$ of the same size, there is a $g$ in $G$ such that $g(A)=B$. 

\subsection*{Oligomorphic Groups with and without a Minimum}

Question~2.3 of \cite{DM} asks which oligomorphic groups have a minimum group topology, mentioning $S(\N)$, the automorphism group of
the countable dense linear order, and the homeomorphism group of the Cantor space, in particular. We present a reasonably broad answer to this question, encompassing the mentioned groups.

Let $Q=\mathbb{Q} \cap (0,1)$. Then $\mathop{Aut}(\mathbb{Q},<)$ (the automorphism group of the countable dense linear order) and $\mathop{Aut}(Q,<)$ are isomorphic. This latter group naturally embeds, say as $G$, in $H([0,1])$ (perhaps this is most clear if we think of elements of $[0,1]$ as Dedekind cuts of $Q$). Further $G$ clearly satisfies the `locally non-trivial' condition of Theorem~\ref{thm:taukc_min_GEN}. Hence ($G$ and its isomorph) $\mathop{Aut}(Q,<)$ has the minimum Hausdorff group topology, $\tau_m$. It is not difficult to see that in this case $\tau_m$ has the following sets as a neighborhood of the identity: $B_\epsilon = \{\alpha \in \mathop{Aut}(Q,<) : \forall x \in Q \ d(\alpha (x),x) < \epsilon\}$, where $d$ is the usual metric on $\R$. No set of this form, $B_\epsilon$, is a subset of $\{\alpha \in \mathop{Aut}(Q,<) : \alpha(1/2)=1/2\}$ which is open in the pointwise topology, $\tau_p$, on $\mathop{Aut}(Q,<)$. Hence $\tau_m$ is strictly contained in $\tau_p$, and so $\tau_p$ is not minimal.

Let $S^1_Q$ -- the `rational circle' -- be any countable dense subset of the circle, $S^1$. Similar considerations apply to the following oligomorphic groups:  all order preserving or order reversing bijections of $Q$, all bijections of the rational circle  which preserve the cyclic order and  all bijections of the rational circle which preserve or reverse the cyclic order. They are all oligomorphic. They all embed either in $H(I)$ or $H(S^1)$, satisfy the `locally non-trivial' condition of Theorem~\ref{thm:taukc_min_GEN}, and so  they have a minimum Hausdorff group topology which is easily seen to be strictly finer than the pointwise topology.

Recall that for $S(\N)$ the topology of pointwise convergence is the minimum Hausdorff group topology \cite{Gau}, and note $S(\N)$ is the only automorphism group which is transitive.

Cameron showed \cite{Cam1} that the highly homogeneous non-transitive  automorphism groups as precisely those automorphism groups listed above. Hence we can summarize the above observations as follows.

\begin{thm}\label{hom_min}
Every highly homogeneous (oligomorphic) automorphism group $G$
has a minimum group topology.

That minimum group topology is strictly coarser than $\tau_p$, except when $G = S(\N)$, and so $\tau_p$ is not minimal.
\end{thm}

However not all oligomorphic groups have a minimum Hausdorff group topology. Indeed let $B$ denote the (unique) atomless countable Boolean algebra. Then $\mathop{Aut}(B)$ is oligomorphic and is well known to be isomorphic to the homeomorphism group of the Cantor set, $H(C)$. Since we have proved, Corollary~\ref{nomin_MCH}, that this group does not have a minimum Hausdorff group topology, we have:
\begin{eg} The oligomorphic group
$\mathop{Aut}(B)$ has no minimum Hausdorff group  topology.
\end{eg}

\end{document}